\documentclass[11pt]{amsart}

\usepackage[
    paper=a4paper, 
    margin=2.54cm
]{geometry}

\usepackage{url}
\usepackage{amsthm}
\usepackage{amsmath}
\usepackage{amssymb}
\usepackage{mathtools}
\usepackage{mathrsfs}
\usepackage{thmtools}
\usepackage[breaklinks,hidelinks]{hyperref}
\usepackage[capitalise]{cleveref}

\theoremstyle{definition}
\newtheorem{thm}{T}
\numberwithin{thm}{section}
\numberwithin{equation}{section}
\newtheorem{definition}[thm]{Definition}
\newtheorem{theorem}[thm]{Theorem}
\newtheorem{lemma}[thm]{Lemma}
\newtheorem{corollary}[thm]{Corollary}
\newtheorem{proposition}[thm]{Proposition}
\newtheorem{remark}[thm]{Remark}
\newtheorem{example}[thm]{Example}

\usepackage{enumitem}
\setlist[itemize]{noitemsep, nolistsep}
\setlist[itemize,1]{label=\textbullet}
\setlist[itemize,2]{label=\(\circ\)}
\makeatletter
\def\namedlabel#1#2{\begingroup
    #2%
    \def\@currentlabel{#2}%
    \phantomsection\label{#1}\endgroup
}
\makeatother

\newcommand{\Define}[1]{\emph{#1}}
\newcommand{\emphasis}[1]{\textit{#1}}
\newcommand{\setof}[2]{\{ #1 \mid #2 \}}
\newcommand{\compl}[1]{#1^c}
\newcommand{\Z}{\mathbb{Z}}
\newcommand{\N}{\mathbb{N}}
\newcommand{\range}{\mathbf{r}}
\newcommand{\source}{\mathbf{s}}
\newcommand{\degree}{\mathbf{d}}
\newcommand{\morphdom}[1]{\mathrm{dom}(#1)}
\newcommand{\principal}[1]{F(#1)}
\newcommand{\PSfiltershiftoff}[2]{#2 \cdot #1}
\newcommand{\PSfiltershifton}[2]{#1 #2}
\newcommand{\ofFMY}{\mathrm{FMY}}
\newcommand{\ofYeend}{\mathrm{Yee}}
\newcommand{\ofSpielberg}{\mathrm{Spi}}
\newcommand{\ofOP}{\mathrm{OP}}
\newcommand{\M}{\mathrm{m}}
\newcommand{\F}{\mathrm{f}}
\newcommand{\PS}[2]{X_{#2, #1}}
\newcommand{\PG}[2]{{G_{#2, #1}}}
\newcommand{\BPS}[2]{\partial{#2}_{#1}}
\newcommand{\BPG}[2]{\partial\PG{#1}{#2}}
\newcommand{\tightgpd}{G_\mathrm{tight}}
\newcommand{\RPA}[2]{#1 \cdot #2}
\newcommand{\RPAdom}[2]{\textrm{dom}_{#1}(#2)}
\newcommand{\RPAcod}[2]{\textrm{im}_{#1}(#2)}
\newcommand{\RPAmap}[2]{{T_{#1}^{#2}}}
\newcommand{\SDPG}[1]{G_{#1}}
\newcommand{\SDPGcylinder}[5]{Z_{#1}(#2,#3,#4,#5)}
\newcommand{\reduction}[2]{#1|_{#2}}
\newcommand{\composablepairs}[1]{#1^{(2)}}
\newcommand{\unitspace}[1]{#1^{(0)}}
\newcommand{\Fell}[3]{Z_{#1}\left(#2 \setminus #3\right)}
\newcommand{\Fellin}[2]{Z_{#1}\left(#2\right)}
\newcommand{\Fellgpd}[5]{Z_{#1}\left(#2 \setminus #3, #4 \setminus #5\right)}

\begin{document}

\title[
    Groupoids of finitely aligned higher-rank graphs
]{
    Groupoids of finitely aligned higher-rank graphs \\
    via filters and graph morphisms
}

\author[L. O. Clark]{Lisa Orloff Clark}
\address[L. O. Clark]{
    School of Mathematics and Statistics,
    Victoria University of Wellington,
    PO Box 600,
    Wellington 6140,
    New Zealand
}
\email{\href{mailto:lisa.orloffclark@vuw.ac.nz}{lisa.orloffclark@vuw.ac.nz}}

\author[M. Jones]{Malcolm Jones}
\address[M. Jones]{
    Institute of Mathematics, Physics and Mechanics,
    Ljubljana 1000,
    Slovenia
}
\email{\href{mailto:malcolm.jones@imfm.si}{malcolm.jones@imfm.si}}

\begin{abstract}
Higher-rank graphs are certain small categories that are a higher-dimensional generalisation of both directed graphs and free monoids.
Path and boundary-path groupoids of finitely aligned higher-rank graphs are often constructed using either filters or graph morphisms. 
We generalise the graph morphism approach to finitely aligned $P$-graphs where $(Q, P)$ is a weakly quasi-lattice ordered group, and we show the filter approach and the graph morphism approach yield isomorphic path and boundary-path groupoids. 
To do this, we define conjugacy of partial monoid actions such that conjugate actions have isomorphic semidirect product groupoids. 
Combining our results with others in the literature, we survey many isomorphic presentations of path and boundary-path groupoids at different levels of generality.
\end{abstract}

\maketitle

\section{Introduction}

A higher-rank graph is a small category $\Lambda$ together with a functor $\degree \colon \Lambda \to \N^k$, where $k$ is a positive integer, satisfying the \Define{unique factorisation property}: for every $\lambda \in \Lambda$ and $m, n \in \N^k$ with $\degree(\lambda) = m + n$, there are unique elements $\mu, \nu \in \Lambda$ such that $\lambda = \mu\nu$, $\degree(\mu) = m$ and $\degree(\nu) = n$.
Higher-rank graphs were introduced in \cite{KP00} as higher-dimensional generalisations of directed graphs in order to broaden the scope of the graph approach to C*-algebras of \cite{KPRR97}.
In fact, their impact goes beyond C*-algebras.
For example, one-vertex higher-rank graphs generalise free monoids \cite{LV24_free}, and generalisations of Thompson groups can be realised via groupoids and inverse semigroups associated to higher-rank graphs \cite{LV20,LSV24}.
This paper grew out of a question raised by Exel about such groupoids and inverse semigroups \cite[Page 195]{Exe08}:
\begin{quote}
Although we have not invested all of the necessary energy to study the inverse semigroup constructed from a general higher rank graph, as in \cite{FMY05}, we conjecture that the groupoid there is the same as the groupoid $\tightgpd$\dots
\end{quote}

In this paper, we show various constructions do indeed yield isomorphic groupoids. 
Given a finitely aligned higher-rank graph $\Lambda$, we combine our main result with a result from \cite{OP20} to conclude that Exel's tight groupoid $\tightgpd$ of the inverse hull of $\Lambda$ is isomorphic to the boundary-path groupoid of \cite{FMY05} (see \cref{rem:gpd_isoms_for_kgraphs}). 
Further analysis of the relationship between the inverse hull and the inverse semigroup of \cite{FMY05} is needed to determine whether $\tightgpd$ of the inverse semigroup of \cite{FMY05} is isomorphic too, which would be the most complete answer to Exel's conjecture. 
In this sense, we give a partial answer to Exel's conjecture in this paper. 

There are two main approaches to defining the path space of a higher-rank graph $\Lambda$. 
The first, initiated in \cite{KP00} for the row-finite case and generalised in \cite{Yee07} to the finitely aligned case, is to use certain graph morphisms from path prototypes to $\Lambda$. 
The second is to use filters of $\Lambda$, which stems from various work \cite{Nic92,Exe08,Spi12} and was first applied to higher-rank graphs in \cite{BSV13}. 
Similar to \cite{ACaHJL}, where the groupoid of filters and the groupoid of germs of an inverse semigroup are shown to be isomorphic, the main result in this paper is that the filter and graph morphism approaches yield isomorphic groupoids (\cref{cor:isom}), which we use to answer Exel's conjecture. 
This is folklore, and evidence can be found in \cite[Remarks 2.2]{KP00}, \cite[Page 2763]{PW05} and \cite[Proposition 19.11]{Exe08}, but the complete details have not appeared.
Moreover, we extend this framework by working more generally with the finitely aligned $P$-graphs of \cite{BSV13} for weakly quasi-lattice ordered groups $(Q, P)$ in the sense of \cite{aHNSY21}. 
Higher-rank graphs, also called $k$-graphs, correspond to the special case where $(Q, P) =(\Z^k, \mathbb{N}^k)$. 
The generalisation of $k$-graphs considered in \cite{CKSS14}, where $\N^k$ is replaced with a finitely-generated cancellative abelian monoid $P$, is outside the scope of this paper.

Our preferred approach is to define the path space of a $P$-graph using filters, as in \cite{Jon26}. 
Even for $k$-graphs, the appeal of this approach becomes clear when the graph is not row-finite. 
In that setting, the graph morphisms approach becomes a lot more complicated. 
See for example \cite[Example~5.4]{CP17_Cohn} for the path groupoid of a row-finite $k$-graph with no sources and \cite[Example~5.2]{CP17_KP} for the boundary-path groupoid of a finitely aligned $k$-graph.
In what follows, we introduce a graph morphism approach to define the path and boundary-path spaces for $P$-graphs that is new (and messy). 
Then we show it is isomorphic to the cleaner filter version. 

To prove our main result that the path groupoids via filters are isomorphic to the path groupoids via graph morphisms for finitely aligned $P$-graphs, where $(Q, P)$ is a weakly quasi-lattice ordered group, we extend the definition of \emphasis{conjugacy} for Deaconu--Renault systems from \cite{ABCE23} to the (partial) monoid actions described in \cite{RW17} (therein called ``semigroup actions'').
We prove that if two monoid actions are conjugate, then they have topologically isomorphic semidirect product groupoids.
Various notions of conjugacy of special $k$-graphs have been studied in \cite{CR2021}. 
It would be interesting to interpret these ideas in the framework of the monoid actions and semidirect product groupoids of \cite{RW17}, but this is not considered in the present paper.

The \emphasis{topological} higher-rank graphs of \cite{Yee07, RW17, RSWY18} are beyond the scope of the present paper. 
That is, we consider only finitely aligned \emphasis{discrete} $P$-graphs. 
In order to show our path and boundary-path groupoids of finitely aligned (discrete) $P$-graphs generalise those of $k$-graphs, we looked for a description of both groupoids in the discrete setting in the literature. 
However, descriptions of the groupoids in the discrete setting are restrictive in one way or another:
\begin{itemize}
\item \cite[\S 3]{Web11} describes the path \emphasis{space} of finitely aligned $k$-graphs but not the \emphasis{groupoid};
\item \cite[Appendix B.1]{aHKR15} describes the path groupoid of \emphasis{row-finite} $k$-graphs;
\item \cite[Example 5.4]{CP17_Cohn} describes the path groupoid of \emphasis{row-finite} $k$-graphs with \emphasis{no sources}; and
\item \cite[Example 5.2]{CP17_KP} describes only the \emphasis{boundary}-path groupoid of finitely aligned $k$-graphs.
\end{itemize}
We believe it is necessary to use \cite{Yee07} as our reference for path and boundary-path groupoids of finitely aligned discrete $k$-graphs.

After a preliminaries section, in \cref{section:conjugacy} we introduce conjugacy for monoid actions and show in \cref{thm:conjugate_actions_have_isomorphic_groupoids} that conjugate actions yield isomorphic partial action groupoids. 
We then describe the filter and graph morphism approaches to path and boundary-path groupoids of finitely aligned $P$-graphs in \cref{section:graphs} and show in \cref{sec:conjugacy_of_filters_and_morphisms} that the corresponding monoid actions are conjugate, see \cref{thm:conjugacy}. 
Consequently, the path and boundary-path groupoids via filters and graph morphisms are isomorphic, see \cref{cor:isom}. 
For finitely aligned $P$-graphs, we reconcile with Spielberg's groupoids from \cite{Spi20} and remark that the boundary-path groupoid has an inverse semigroup model due to a result of \cite{OP20}, see \cref{rem:gpd_isoms_for_Pgraphs}. 
In \cref{sec:k}, we show that the path and boundary-path groupoids coincide with those of \cite{FMY05, Yee07} for finitely aligned (discrete) $k$-graphs, see \cref{rem:gpd_isoms_for_kgraphs}. 
We conclude the paper with \cref{rem:Exel}, where we respond to Exel's conjecture. 
It is the nature of this paper that we construct/recall many groupoids from the literature. 
Their notation is inconsistent in the literature, so we summarise the convention we have chosen for this paper in \cref{sec:notation}.

\section{Preliminaries}
\label{sec:prelims}

\subsection{Partial monoid actions}

The following definitions come from \cite[\S 5]{RW17}, but we use the term \Define{monoid}, i.e. a semigroup with an identity element.
Let $X$ be a set, and let $P$ be a monoid whose identity element is denoted by $e$.
Suppose $X * P$ is a subset of $X \times P$ and $T_X \colon X * P \to X$ is a function sending each $(x,m) \in X * P$ to $T_X((x, m)) \eqqcolon \RPA{x}{m} \in X$ and satisfying:
\begin{enumerate}
\item[\namedlabel{S1}{(S1)}] for all $x \in X$, $(x,e) \in X * P$ and $\RPA{x}{e} = x$; and
\item[\namedlabel{S2}{(S2)}] for all $(x,m,n) \in X \times P \times P$, $(x,mn) \in X * P$ if and only if $(x,m) \in X * P$ and $(\RPA{x}{m},n) \in X * P$, in which case $\RPA{(\RPA{x}{m})}{n} = \RPA{x}{(mn)}$.
\end{enumerate}
We call the triple $(X, P, T_X)$ a \Define{(partial) monoid action (of $P$ on $X$)}.
For each $m \in P$, we write $\RPAdom{X}{m} \coloneqq \setof{x \in X}{(x, m) \in X * P}$ and $\RPAcod{X}{m} \coloneqq \setof{\RPA{x}{m}}{x \in \RPAdom{X}{m}}$, and we define $\RPAmap{X}{m} \colon \RPAdom{X}{m} \to \RPAcod{X}{m}$ by $\RPAmap{X}{m}(x) = \RPA{x}{m}$, for each $x \in \RPAdom{X}{m}$. 
Whenever we write $\RPA{x}{m}$, it is to be understood that $x \in \RPAdom{X}{m}$.
The monoids $P$ that we consider come with a partial order $\leq$ (see \cref{sec:WQLO}). 
In this case, we say $(X, P, T_X)$ is \Define{directed} if, for all $m,n \in P$, $\RPAdom{X}{m} \cap \RPAdom{X}{n} \neq \emptyset$ implies there is $l \in P$ such that $m,n \leq l$ and $\RPAdom{X}{m} \cap \RPAdom{X}{n} \subseteq \RPAdom{X}{l}$, in which case $\RPAdom{X}{m} \cap \RPAdom{X}{n} = \RPAdom{X}{l}$ as per the note following \cite[Definition 5.2]{RW17}.
Recall that a \Define{local homeomorphism} from a space $X$ to a space $Y$ is a continuous map $f \colon X \to Y$ such that every $x \in X$ has an open neighbourhood $U$ such that $f(U)$ is open in $Y$ and the restriction $f|_U \colon U \to f(U)$ of $f$ is a homeomorphism \cite[Page 68]{Wil70}.
A monoid action $(X, P, T)$, where $P$ is a submonoid of a discrete group $Q$, is \Define{locally compact} if $X$ is a locally compact Hausdorff space and, for all $m \in P$, $\RPAdom{X}{m}$ and $\RPAcod{X}{m}$ are open in $X$ and $\RPAmap{X}{m} \colon \RPAdom{X}{m} \to \RPAcod{X}{m}$ is a local homeomorphism.
In order to characterise invariant subsets of the unit space of the semidirect product groupoids associated to monoid actions, we make the following definition.

\begin{definition}
\label{def:invariant}
Given a monoid action $(X, P, T_X)$, we say a subset $U \subseteq X$ is \Define{invariant (with respect to $T_X$)} if, for all $x \in X$, $y \in U$ and $m, n \in P$, $\RPA{x}{m} = \RPA{y}{n}$ implies $x \in U$.
\end{definition}

\subsection{Groupoids}
\label{sec:groupoids}

We refer the reader to \cite{Pat99,Wil19,SSW20} for the following preliminaries on groupoids.
A \Define{groupoid} consists of a set $G$ and a subset $\composablepairs{G} \subseteq G \times G$ together with a map $(g,h) \mapsto gh$ from $\composablepairs{G}$ to $G$ (called \Define{composition}) and a map $g \mapsto g^{-1}$ from $G$ to $G$ (called \Define{inversion}) such that:
\begin{itemize}
\item if $(f, g),(g, h) \in \composablepairs{G}$, then $(f, gh),(fg, h) \in \composablepairs{G}$ and $f(gh) = (fg)h \eqqcolon fgh$;
\item for all $g \in G$, $(g^{-1})^{-1} = g$; and
\item for all $g \in G$, we have $(g^{-1},g) \in \composablepairs{G}$, and if $(g,h) \in \composablepairs{G}$, then $g^{-1}(gh) = h$ and $(gh)h^{-1} = g$.
\end{itemize}
The set $\unitspace{G} \coloneqq \setof{x \in G}{x = xx = x^{-1}}$ is called the \Define{unit space} of $G$.
There are \Define{range} and \Define{source} maps from $G$ to $\unitspace{G}$ given by $\range(g) \coloneqq gg^{-1}$ and $\source(g) \coloneqq g^{-1}g \in \unitspace{G}$, for each $g \in G$, respectively.
We have $\range(G) = \source(G) = \unitspace{G}$, and $(g, h) \in \composablepairs{G}$if and only if $\source(g) = \range(h)$.

Let $G$ be a groupoid endowed with a (\emphasis{not necessarily Hausdorff}) topology.
The set $\composablepairs{G}$ is endowed with the subspace topology from the product topology on $G \times G$.
We say $G$ is \Define{topological} if composition and inversion are continuous.
A groupoid $G$ is \Define{\'etale} if $\source \colon G \to \unitspace{G}$ is a local homeomorphism. 
By \cite[Theorem 5.18]{Res07}, a topological groupoid $G$ is \'etale if and only if $\unitspace{G}$ is open in $G$ and $\source$ is an open map.
Given an \'etale locally compact (not necessarily Hausdorff) groupoid $G$ such that $\unitspace{G}$ is Hausdorff, we say $G$ is \Define{ample} if $\unitspace{G}$ has a basis of compact open sets (which is equivalent to the definition via bisections by \cite[Proposition 3.6]{Ste10} or \cite[Proposition~4.1]{Exe2010}).

If $G$ is a groupoid, we call $H \subseteq G$ a \Define{subgroupoid} if $H$ is a groupoid under the inherited operations.
Given $U \subseteq \unitspace{G}$, the \Define{reduction} of $G$ to $U$ is $\reduction{G}{U} \coloneqq \source^{-1}(U) \cap \range^{-1}(U)$, which is a subgroupoid of $G$.
We say $U \subseteq \unitspace{G}$ is \Define{invariant} if $\range(\source^{-1}(U)) \subseteq U$, in which case $\reduction{G}{U} = \source^{-1}(U)$.
Given an \'etale locally compact groupoid $G$ with Hausdorff unit space $\unitspace{G}$, if $U \subseteq \unitspace{G}$ is closed and invariant, then $\reduction{G}{U}$ is an \'etale locally compact closed subgroupoid of $G$ (see \cite[Page 96]{SSW20}).
Given topological groupoids $G$ and $H$, a \Define{(topological) isomorphism} from $G$ to $H$ is a homeomorphism $\psi \colon G \to H$ such that $(g_1, g_2) \in \composablepairs{G}$ if and only if $(\psi(g_1), \psi(g_2)) \in \composablepairs{H}$, in which case $\psi(g_1g_2) = \psi(g_1)\psi(g_2)$, and $\psi(g^{-1}) = \psi(g)^{-1}$ for all $g \in G$. 
We say $G$ and $H$ are \Define{(topologically) isomorphic} if there is a (topological) isomorphism from $G$ to $H$.

\subsection{Semidirect product groupoids}
\label{sec:semidirect_product_groupoid_prelims}

Semidirect product groupoids are a generalisation of Deaconu--Renault groupoids \cite{Dea95, Ren00}. 
Deaconu--Renault groupoids are used in \cite[\S 2.2]{ABCE23}.
Let $(X, P, T_X)$ be a directed and locally compact monoid action, where $P$ is a submonoid of a discrete group $Q$.
In \cite{RW17}, Renault and Williams associate a semidirect product groupoid $\SDPG{X}$ to $(X, P, T_X)$ as follows.
Let $\SDPG{X}$ be the set of $(x,q,y) \in X \times Q \times X$ for which there are $m,n \in P$ satisfying $q = mn^{-1}$, $x \in \RPAdom{X}{m}$, $y \in \RPAdom{X}{n}$ and $\RPA{x}{m} = \RPA{y}{n}$. 
Let $\composablepairs{\SDPG{X}}$ be the set of pairs $((x,q,y),(w,r,z)) \in \SDPG{X} \times \SDPG{X}$ with $y = w$. 
Then, $\SDPG{X}$ is a groupoid with composition 
\(
(x,q,y)(y,r,z) \coloneqq (x,qr,z)
\)
and inversion 
\(
(x,q,y)^{-1} \coloneqq (y,q^{-1},x),
\)
called the \Define{semidirect product groupoid} of $(X, P, T_X)$.
The collection of 
\[
\SDPGcylinder{X}{U}{m}{n}{V} \coloneqq 
\setof{
(x,mn^{-1},y) \in \SDPG{X}
}{
x \in U, y \in V \text{ and } \RPA{x}{m} = \RPA{y}{n}
},
\]
where $m, n \in P$ and $U, V \subseteq X$ are open, is a basis for a topology on $\SDPG{X}$ such that $\SDPG{X}$ is an \'etale locally compact Hausdorff groupoid, and $x \mapsto (x, e, x)$ is a homeomorphism from $X$ to $\unitspace{\SDPG{X}}$.
Thus, if $X$ has a basis of compact open sets, then $\SDPG{X}$ is ample.
Also, given any basis for the topology on $X$, the collection of $\SDPGcylinder{X}{U}{m}{n}{V}$, where $m, n \in P$ and $U, V \subseteq $ are basic open sets, is a basis for the topology on $\SDPG{X}$.
The following lemma is straight-forward so we omit the proof.
It allows us to identify invariant subsets of $X$ in the sense of \cref{def:invariant} with invariant subsets of $\unitspace{\SDPG{X}}$ via the homeomorphism $x \mapsto (x, e, x)$ from $X$ to $\unitspace{\SDPG{X}}$.

\begin{lemma}
\label{lem:inv}
Let $(X, P, T_X)$ be a directed and locally compact monoid action, where $P$ is a submonoid of a discrete group $Q$.
\begin{enumerate}
\item \label{it1:leminv} For any $U \subseteq X$, $U$ is invariant with respect to $T_X$ if and only if $U$ identifies with an invariant subset of $\unitspace{\SDPG{X}}$ via $x \mapsto (x, e, x)$.
\item \label{item2:inv} For any closed invariant $U \subseteq X$, we have
\(
\reduction{\SDPG{X}}{U} = \setof{(x, q, y) \in \SDPG{X}}{x, y \in 
U},
\)
which is an \'etale locally compact Hausdorff closed subgroupoid of $\SDPG{X}$.
\end{enumerate}
\end{lemma}

\subsection{Weakly quasi-lattice ordered groups}
\label{sec:WQLO}

We follow \cite{aHNSY21} for preliminaries on weakly quasi-lattice ordered groups.
Let $Q$ be a discrete group, and let $P$ be a monoid such that $P \cap P^{-1} = \{e\}$. 
For each $p, r \in P$, define $p \leq r$ if $pq = r$ for some $q \in P$, which is left invariant.
We say $(Q, P)$ is a \Define{weakly quasi-lattice ordered group} if, whenever two elements of $P$ have a common upper bound with respect to $\leq$, there is a least common upper bound.
Roughly speaking, the $P$-graphs introduced in \cite{BSV13} and considered in this paper are defined by replacing the role of $(\Z^k, \N^k)$ in the $k$-graphs of \cite{KP00} with a weakly quasi-lattice ordered group $(Q, P)$. 
A different generalisation of $k$-graphs is considered in \cite{CKSS14}, where $\N^k$ is replaced with a finitely-generated cancellative abelian monoid. 
We give examples of a weakly quasi-lattice ordered group that is not finitely-generated and a finitely-generated cancellative abelian monoid that is not weakly quasi-lattice ordered to emphasise that neither collection of generalised $k$-graphs is contained in the other. 
In this paper, we consider the $P$-graphs of \cite{BSV13}.

\begin{example}[Example 3.2 of \cite{aHNSY21} and Section 2.3 of \cite{Nic92}]
Let $\mathbb{F}$ be the free group on a countably infinite set of generators, and let $\mathbb{F}^+$ be the smallest submonoid containing the identity and the generators. 
Then, $(\mathbb{F}, \mathbb{F}^+)$ is weakly quasi-lattice ordered (in fact, it is a quasi-lattice in the sense of \cite{aHNSY21}), but $\mathbb{F}^+$ is not finitely-generated.
\end{example}

\begin{example}
Let $P$ be the smallest submonoid of the additive monoid $\N^2$ containing $(1, 0)$, $(1, 1)$ and $(1, 2)$, and suppose $P$ embeds in a discrete group $Q$.
Observe $(1, 0) + (1, 1) = (2, 1)$ and $(1, 0) + (1, 2) = (1, 1) + (1, 1) = (2, 2)$, so $(2, 1)$ and $(2, 2)$ are distinct minimal common upper bounds of $(1, 0)$ and $(1, 1)$. 
Thus, $(1, 0)$ and $(1, 1)$ do not have a least common upper bound, and so $(Q, P)$ is not weakly quasi-lattice ordered.
\end{example}

\subsection{\texorpdfstring{$P$-graphs}{P-graphs}}
\label{sec:Pgraphs}

Following \cite{BSV13}, given a weakly quasi-lattice ordered group $(Q, P)$, a (discrete) \Define{$P$-graph} consists of a countable small category $\Lambda$ with composable pairs $\composablepairs{\Lambda}$, units $\unitspace{\Lambda}$, range map $\range$ and source map $\source$, together with a functor $\degree \colon \Lambda \to P$ satisfying the \Define{unique factorisation property}: for every $\lambda \in \Lambda$ and $p, q \in P$ with $\degree(\lambda) = pq$, there are unique elements $\mu, \nu \in \Lambda$ such that $\lambda = \mu\nu$, $\degree(\mu) = p$ and $\degree(\nu) = q$.
For each $p \in P$, we write $\Lambda^p \coloneqq \degree^{-1}(p)$, and we have $\unitspace{\Lambda} = \Lambda^e$.
For each $\mu \in \Lambda$ and $E \subseteq \Lambda$, we write 
\(
\lambda E \coloneqq \setof{
\lambda\mu
}{
\mu\in E \text{ and } \source(\lambda) = \range(\mu)
}.
\)
We say $\Lambda$ is \Define{finitely aligned} if, for all $\mu, \nu \in \Lambda$, there is a finite $J \subseteq \Lambda$ such that $\mu\Lambda \cap \nu\Lambda = \bigcup_{\lambda \in J} \lambda\Lambda$ as in \cite[Page 729]{Spi12}. 
For each $\mu, \lambda \in \Lambda$, define $\mu \preceq \lambda \iff \mu\nu = \lambda \text{ for some } \nu \in \Lambda \iff \lambda \in \mu\Lambda$.

We state some properties and notations for $P$-graphs.
It follows from the unique factorisation property that $\Lambda$ is \Define{left cancellative} (i.e. $\mu\nu = \mu\kappa$ implies $\nu = \kappa$), \Define{right-cancellative} (i.e. $\mu\nu = \kappa\nu$ implies $\mu = \kappa$), and has \Define{no inverses} (i.e. $\mu\nu = \source(\nu)$ implies $\mu = \nu = \source(\nu)$). 
Hence, any $P$-graph is a \Define{category of paths} (i.e. a left and right cancellative small category having no inverses) as in \cite[Example 2.2(6)]{Spi14}.
The relation $\preceq$ is the same as that of \cite[Definition 2.5]{Spi14} for categories of paths, so $\preceq$ is a partial order. 
Moreover, $\preceq$ is \Define{left invariant} (i.e. $\mu\nu \preceq \mu\kappa$ implies $\nu \preceq \kappa$) because of left cancellation.
The $k$-graphs of \cite{KP00} are the $P$-graphs where $(Q, P) = (\mathbb{Z}^k, \N^k)$ for some $k \in \mathbb{Z}^+$.
A \Define{$P$-graph morphism} is a functor $x \colon \Lambda_1 \to \Lambda_2$ from a $P$-graph $(\Lambda_1, \degree_1)$ to a $P$-graph $(\Lambda_2, \degree_2)$ satisfying $\degree_1(\lambda) = \degree_2(x(\lambda))$, for all $\lambda \in \Lambda_1$.
For discrete $\Lambda$, the notion of ``rank-$k$ (finite) path prototypes'' of \cite{Yee07} generalises to weakly quasi-lattice ordered groups $(Q, P)$ without modification as follows.

\begin{example}
\label{ex:path_prototypes_finite}
Let $(Q, P)$ be a weakly quasi-lattice ordered group. 
For each $m \in P$, the set
\[
\Omega_{P, m} \coloneqq 
\setof{
(p, q) \in P \times P
}{
p \leq q \leq m
}
\]
is a $P$-graph with $(p, q), (q', r)\in\composablepairs{\Omega_{P, m}}$ if and only if $q = q'$ and operations
\[
\composablepairs{\Omega_{P, m}} \to \Omega_{P, m}, \quad ((p, q), (q, r)) \mapsto (p, r),
\]
and
\[
\degree \colon \Omega_{P, m} \to P, \quad (p, q) \mapsto p^{-1}q.
\]
\end{example}

In this paper, we reconcile constructions of the path space, boundary-path space, path groupoid and boundary-path groupoid of a $P$-graph. 
There are a number of different constructions that yield the same structures up to isomorphism. 
The notation of these structures is inconsistent in the literature. 
To make this paper as readable as possible, we work with a particular convention, which we collate in \cref{sec:notation}.

\section{Conjugacy of monoid actions}
\label{section:conjugacy}

We generalise the definition of a conjugacy of Deaconu--Renault systems as in \cite{ABCE23} via a characterisation in \cite[Lemma 2.6(2)]{ABCE23}.
A number of questions can be asked about conjugacy of monoid actions in analogy with \cite[Theorem 3.1]{ABCE23}. 
In this paper, we only check that conjugate monoid actions have isomorphic semidirect product groupoids 
(\cref{thm:conjugate_actions_have_isomorphic_groupoids}) to provide ourselves with a framework to show that different constructions of groupoids from $P$-graphs are isomorphic (\cref{cor:isom}).
Various notions of conjugacy of finitely aligned $k$-graphs are studied in \cite{CR2021}. 
In particular, two finitely aligned $k$-graphs are shown to be eventually one-sided conjugate (see \cite[Definition 3.1]{CR2021}) if and only if there is a cocycle-preserving isomorphism between their associated boundary-path groupoids \cite[Theorem 3.2]{CR2021}. 
It would be interesting to investigate an analogue of eventual one-sided conjugacy for monoid actions to characterise when two semidirect product groupoids admit a cocycle-preserving isomorphism.

\begin{definition}
\label{def:conjugate}
Let $P$ be a monoid, and let $(X, P, T_X)$ and $(Y, P, T_Y)$ be monoid actions on topological spaces $X$ and $Y$.
A \Define{conjugacy} is a map $h \colon X \to Y$ such that:
\begin{itemize}
\item[\namedlabel{C1}{(C1)}] $h$ is a homeomorphism;
\item[\namedlabel{C2}{(C2)}] $h(\RPAdom{X}{m}) = \RPAdom{Y}{m}$, for all $m \in P$; and 
\item[\namedlabel{C3}{(C3)}] $h(\RPA{x}{m}) = \RPA{h(x)}{m}$, for all $m \in P$ and $x \in \RPAdom{X}{m}$.
\end{itemize}
We say $(X, P, T_X)$ and $(Y, P, T_Y)$ are \Define{conjugate} if there is a conjugacy $h \colon X \to Y$.
\end{definition}

It follows from the definition of a conjugacy that $h(\RPAcod{X}{m}) = \RPAcod{Y}{m}$.

\begin{proposition}
\label{prop:conjugacy_preserves_properties}
Let $P$ be a monoid, and let $(X, P, T_X)$ and $(Y, P, T_Y)$ be conjugate monoid actions. 
\begin{enumerate}
\item If $(X, P, T_X)$ is locally compact, then $(Y, P, T_Y)$ is locally compact. 
\item If $(X, P, T_X)$ is directed, then $(Y, P, T_Y)$ is directed.
\end{enumerate}
\end{proposition}
\begin{proof}
Say $h \colon X \to Y$ is a conjugacy, and suppose $(X, P, T_X)$ is locally compact.
Since $X$ is locally compact Hausdorff, \ref{C1} implies $Y$ is locally compact Hausdorff.
Fix $m \in P$.
Since $\RPAdom{X}{m}$ is open, \ref{C1} and \ref{C2} together imply $\RPAdom{Y}{m}$ is open.
Also, $\RPAcod{X}{m}$ is open, and $h(\RPAcod{X}{m}) = \RPAcod{Y}{m}$, so $\RPAcod{Y}{m}$ is open too.
Since $h$ and $h^{-1}$ are homeomorphisms, $T^m_X$ is a local homeomorphism and \ref{C3} says $T_Y^m=h \circ T_X^m \circ h^{-1}$, so we get that $\RPAmap{Y}{m} \colon \RPAdom{Y}{m} \to \RPAcod{Y}{m}$ is also local homeomorphism.

Now suppose $(X, P, T_X)$ is directed.
Let $m, n \in P$ be such that $\RPAdom{Y}{m} \cap \RPAdom{Y}{n} \neq \emptyset$.
Because $h$ is injective and because of \ref{C2}, we have
\[
h(\RPAdom{X}{m} \cap \RPAdom{X}{n})
= \RPAdom{Y}{m} \cap \RPAdom{Y}{n},
\]
so $\RPAdom{X}{m} \cap \RPAdom{X}{n} \neq \emptyset$.
Since $(X, P, T_X)$ is directed, there is $l \in P$ such that $m, n \leq l$ and $\RPAdom{X}{m} \cap \RPAdom{X}{n} = \RPAdom{X}{l}$.
Again, because $h$ is injective and because of \ref{C2}, we have
\(
\RPAdom{Y}{m} \cap \RPAdom{Y}{n}
= h(\RPAdom{X}{m}) \cap h(\RPAdom{X}{n})
= h(\RPAdom{X}{m} \cap \RPAdom{X}{n}) 
= h(\RPAdom{X}{l}) 
= \RPAdom{Y}{l}.
\)
\end{proof}

In \cite[Theorem 3.1]{ABCE23}, conjugacy of Deaconu--Renault systems is characterised in terms of groupoid isomorphisms and C*-algebra *-isomorphisms.
We generalise a small part of this theorem via routine arguments alluded to in the proof of \cite[Proposition 3.8]{ABCE23}.

\begin{theorem}
\label{thm:conjugate_actions_have_isomorphic_groupoids}
Let $Q$ be a group, let $P$ be a submonoid, and let $(X, P, T_X)$ and $(Y, P, T_Y)$ be locally compact and directed monoid actions.
If $h \colon X \to Y$ is a conjugacy, then 
\[
\psi_h \colon \SDPG{X} \to \SDPG{Y}, \quad \psi_h((x, q, y)) 
\coloneqq (h(x), q, h(y)),
\]
is an isomorphism such that, for each open $U, V \subseteq X$ and $m, n \in P$, $\psi_h(\SDPGcylinder{X}{U}{m}{n}{V}) = \SDPGcylinder{Y}{h(U)}{m}{n}{h(V)}$.
Moreover, for each closed invariant $U \subseteq X$, $h(U) \subseteq Y$ is closed and invariant, and $\psi_h$ restricts to an isomorphism from $\reduction{\SDPG{X}}{U}$ to $\reduction{\SDPG{Y}}{h(U)}$.
\end{theorem}
\begin{proof}
Let $(x, mn^{-1}, y) \in \SDPG{X}$, where $m, n \in P$ are such that $\RPA{x}{m} = \RPA{y}{n}$.
By definition of a conjugacy, 
\(
\RPA{h(x)}{m} 
= h(\RPA{x}{m}) 
= h(\RPA{y}{n}) 
= \RPA{h(y)}{n},
\)
and $h(x) \in \RPAdom{Y}{m}$ and $h(y) \in \RPAdom{Y}{n}$, so $\psi_h((x, q, y)) = (h(x), mn^{-1}, h(y)) \in \SDPG{Y}$.
That is, $\psi_h$ is well-defined.
By \ref{C1}, $h$ is a bijection.
If $(h(x), q, h(y)) = (h(x'), q', h(y'))$, then $q=q'$ and  $(x, q, y) = (x', q', y')$ because $h$ is injective. 
Thus $\psi_h$ is injective. 
Fix $(y_1, q, y_2) \in \SDPG{Y}$.
Then, there are $m, n \in P$ such that $\RPA{y_1}{m} = \RPA{y_2}{n}$ and $mn^{-1} = q$.
Since $h$ is surjective, there are $x_1, x_2 \in X$ such that $h(x_1) = y_1$ and $h(x_2) = y_2$, so 
\(
h(\RPA{x_1}{m}) 
= \RPA{h(x_1)}{m} 
= \RPA{h(x_2)}{n} 
= h(\RPA{x_2}{n}).
\)
Since $h$ is injective, $\RPA{x_1}{m} = \RPA{x_2}{n}$, and so $(x_1, q, x_2) \in \SDPG{X}$.
Moreover, $(y_1, q, y_2) = \psi_h(x_1, q, x_2)$. 
Thus, $\psi_h$ is surjective.
By similar arguments, we can show $\psi_h(\SDPGcylinder{X}{U}{m}{n}{V}) = \SDPGcylinder{Y}{h(U)}{m}{n}{h(V)}$, for any open $U, V \subseteq X$ and $m, n \in P$.
The collection of such $\SDPGcylinder{X}{U}{m}{n}{V}$ is a basis and $\psi_h$ is a bijection, so it follows that $\psi_h$ is a homeomorphism.
For any $(x, q, y), (w, r, z) \in \SDPG{X}$, $y = w$ if and only if $h(y) = h(w)$ because $h$ is injective, and so $((x, q, y), (w, r, z)) \in \composablepairs{\SDPG{X}}$ if and only if $(\psi_h((x, q, y)), \psi_h((w, r, z))) \in \composablepairs{\SDPG{Y}}$, in which case 
\begin{equation*}
\begin{aligned}
\psi_h((x, q, y))\psi_h((w, r, z)) 
&= (h(x), qr, h(z)) \\
&= \psi_h((x, qr, z)) \\
&= \psi_h((x, q, y)(w, r, z)).
\end{aligned}
\end{equation*}
Moreover, 
\begin{equation*}
\begin{aligned}
\psi_h((x, q, y))^{-1}
&= (h(x), q, h(y))^{-1} \\
&= (h(y), q^{-1}, h(x)) \\
&= \psi_h(y, q^{-1}, x) \\
&= \psi_h((x, q, y)^{-1}).
\end{aligned}
\end{equation*}
Therefore, $\psi_h$ is an isomorphism.

Now suppose $U \subseteq X$ is closed and invariant.
Fix $h(x) \in Y$, $h(y) \in h(U)$ and $m, n \in P$ such that $\RPA{h(x)}{m} = \RPA{h(y)}{n}$.
By \ref{C3}, $h(\RPA{x}{m}) = h(\RPA{y}{n})$, so $\RPA{x}{m} = \RPA{y}{n}$.
Since $h(y) \in h(U)$, $y \in U$, and so $x \in U$ because $U$ is invariant.
Thus, $h(x) \in h(U)$, so $h(U)$ is invariant.
Also, $h(U)$ is closed because $h$ is a homeomorphism.
By \cref{lem:inv}\eqref{item2:inv}, $\reduction{\SDPG{X}}{U}$ and $\reduction{\SDPG{Y}}{h(U)}$ are \'etale locally compact Hausdorff closed subgroupoids of $\SDPG{X}$ and $\SDPG{Y}$, respectively.
We have that $\psi_h(\reduction{\SDPG{X}}{U}) = \reduction{\SDPG{Y}}{h(U)}$. 
Since $\psi_h$ is an isomorphism that maps the subgroupoid $\reduction{\SDPG{X}}{U}$ to the subgroupoid $\reduction{\SDPG{Y}}{h(U)}$, $\psi_h$ restricts to an isomorphism of these subgroupoids.
\end{proof}

\section{Path groupoids via filters and graph morphisms}
\label{section:graphs}

In \cref{sec:filters}, we recall the construction of path groupoids of finitely aligned $P$-graphs \emphasis{via filters} from \cite{Jon26}.
In \cref{sec:morphisms}, we generalise the construction of path groupoids of finitely aligned $k$-graphs \emphasis{via graph morphisms} (e.g. \cite{Yee07}) to include $P$-graphs, where $(Q, P)$ is a weakly quasi-lattice ordered group.
We note filters and graph morphisms are not the only two notions of paths: In \cite{BY17}, paths are defined as sections of functors induced by a discrete Conduch\'e fibration. 
This approach has been further developed in \cite{Wol23}, where only infinite paths are considered.
We have not examined whether filters and graph morphisms coincide with the paths of \cite{BY17, Wol23} in our setting.

\subsection{Filter approach}
\label{sec:filters}

Filters (or similar) have been used as models of paths in $P$-graphs for some 
time (e.g. \cite{Nic92} for quasi-lattice ordered groups, \cite{Exe08} for 
semigroupoids, \cite{Spi12} for categories of paths, and \cite{BSV13} for 
$P$-graphs---though no groupoid is used).

\subsubsection{Path groupoid}
\label{sec:path_groupoid}

We recall the path and boundary-path groupoids of finitely aligned $P$-graphs 
via filters from \cite{Jon26}.
A \Define{filter} of a $P$-graph $\Lambda$ is a nonempty subset $x$ of 
$\Lambda$ that is \Define{hereditary} (i.e. $\lambda \preceq \mu \in x$ 
implies $\lambda \in x$) and \Define{directed} (i.e. $\mu,\nu \in x$ implies 
there is $\lambda \in x$ such that $\mu,\nu \preceq \lambda$).

\begin{definition}[Path space of filters]
\label{def:PSfilter}
Let $(Q, P)$ be a weakly quasi-lattice ordered group, and let $\Lambda$ 
be a finitely aligned $P$-graph.
We write $\PS{\F}{\Lambda}$ for the set of filters of $\Lambda$.
\end{definition}

The set $\PS{\F}{\Lambda}$ is the first of many presentations of the path 
space of a $P$-graph $\Lambda$ considered in this paper. 
The various notation can be found in \cref{sec:notation}.
Each $x \in \PS{\F}{\Lambda}$ is directed, and so the restriction $\degree|_x$ is 
injective (see \cite[Lemma 6.6(b)]{RW17}).
For each $\lambda \in \Lambda$, we write 
\[
\principal{\lambda} \coloneqq \setof{
\mu \in \Lambda
}{
\mu \preceq \lambda
} \in \PS{\F}{\Lambda}.
\]
For any $x \in \PS{\F}{\Lambda}$, there is a unique element of $x \cap \Lambda^{(0)}$ denoted by $\range(x)$.
Given finite (possibly empty) $K_1, K_2 \subseteq \Lambda$ and $\mu \in \Lambda$, we write
\begin{align*}
\Fell{\PS{\F}{\Lambda}}{K_1}{K_2}
&\coloneqq \setof{
x \in \PS{\F}{\Lambda}
}{
K_1 \subseteq x \subseteq \Lambda \setminus K_2
}, \\
\Fell{\PS{\F}{\Lambda}}{\mu}{K_2}
&\coloneqq \Fell{\PS{\F}{\Lambda}}{\{\mu\}}{K_2}, \text{ and } \\
\Fellin{\PS{\F}{\Lambda}}{\mu}
&\coloneqq \Fell{\PS{\F}{\Lambda}}{\mu}{\emptyset}.
\end{align*}
The collection of such $\Fell{\PS{\F}{\Lambda}}{K_1}{K_2}$ is a basis for a topology on $\PS{\F}{\Lambda}$.
Identifying $\PS{\F}{\Lambda}$ as a subcollection of maps from $\Lambda$ to the discrete space $\{0, 1\}$ realises the topology on $\PS{\F}{\Lambda}$ as the topology of pointwise convergence (also called the `pointwise topology') in the sense of \cite[Definition 42.1]{Wil70}. 
In some settings, including \cite{ACaHJL}, such topologies on subsets of the powerset of a set are called `patch topologies' (see the note about this terminology in \cite[Page 293]{LV21}).
The motivation for the name `pointwise convergence' is explained by \cite[Theorem 42.2]{Wil70}, which simplifies as follows in our setting.
A sequence $(x_n) \subseteq \PS{\F}{\Lambda}$ converges to $x \in \PS{\F}{\Lambda}$ if and only if for each $\lambda \in x$ there is an $N \in \mathbb{N}$ such that $N \leq n \implies \lambda \in x_n$ and for each $\lambda \in \Lambda \setminus x$ there is an $N \in \mathbb{N}$ such that $N \leq n \implies \lambda \in \Lambda \setminus x_n$.
By \cite[Remark 4.2 and Theorem 4.9]{Jon26}, if $\Lambda$ is finitely aligned, then the collection of $\Fell{\PS{\F}{\Lambda}}{\mu}{K}$, where $\mu \in \Lambda$ and $K \subseteq \mu\Lambda$ is finite, is a countable basis of compact open sets, and $\PS{\F}{\Lambda}$ is locally compact and Hausdorff.

By \cite[Lemma 3.4]{BSV13}, for each $\lambda \in \Lambda$, the map
\[
\Fellin{\PS{\F}{\Lambda}}{\lambda} 
\to \Fellin{\PS{\F}{\Lambda}}{\source(\lambda)},
\quad 
x 
\mapsto \PSfiltershiftoff{\lambda}{x} 
\coloneqq \setof{\mu \in \Lambda}{\lambda\mu \in x},
\]
is a bijection with inverse 
\[
\Fellin{\PS{\F}{\Lambda}}{\source(\lambda)} 
\to \Fellin{\PS{\F}{\Lambda}}{\lambda},
\quad 
x 
\mapsto \PSfiltershifton{\lambda}{x} 
\coloneqq \setof{
\zeta \in \Lambda
}{
\zeta \preceq \lambda\mu \text{ for some } \mu \in x
}.
\]
Even if $\Lambda$ is not finitely aligned, $x \mapsto 
\PSfiltershiftoff{\lambda}{x}$ is continuous \cite[Lemma 5.3]{Jon26} 
and its inverse $x \mapsto \PSfiltershifton{\lambda}{x}$ has some 
continuity-like behaviour \cite[Lemma 5.13]{Jon26}.
If $\Lambda$ is finitely aligned, then for each $\lambda \in \Lambda$ the 
subspace $\Fellin{\PS{\F}{\Lambda}}{\lambda}$ has a basis of sets of the form 
$\Fell{\PS{\F}{\Lambda}}{\kappa}{K}$, where $\lambda \preceq \kappa$ and $K 
\subseteq \kappa\Lambda$ is finite, by arguments similar to 
\cite[Lemma 6.2]{Jon26}. 
The image of such a $\Fell{\PS{\F}{\Lambda}}{\kappa}{K}$ under $x \mapsto 
\PSfiltershiftoff{\lambda}{x}$ is open, and so $x \mapsto 
\PSfiltershiftoff{\lambda}{x}$ is an open map.
Therefore, if $\Lambda$ is finitely aligned, then $x \mapsto 
\PSfiltershiftoff{\lambda}{x}$ is a homeomorphism.

Following \cite[Lemma 5.1]{Jon26}, given $(\lambda, \mu) \in 
\composablepairs{\Lambda}$ and $x \in \PS{\F}{\Lambda}$, if $\lambda\mu \in 
x$, then $\PSfiltershiftoff{\mu}{(\PSfiltershiftoff{\lambda}{x})} = 
\PSfiltershiftoff{(\lambda\mu)}{x}$, and if $\source(\mu) = \range(x)$, then 
$\PSfiltershifton{\lambda}{(\PSfiltershifton{\mu}{x})} = 
\PSfiltershifton{(\lambda\mu)}{x}$.
By \cite[Lemma 3.4]{BSV13}, if $\mu \in x \in \PS{\F}{\Lambda}$ and 
$(\lambda, \mu) \in \composablepairs{\Lambda}$, then 
$\PSfiltershifton{\mu}{(\PSfiltershiftoff{\mu}{x})} = x$ and 
$\PSfiltershiftoff{\lambda}{(\PSfiltershifton{\lambda}{x})} = x$.
As per \cite[Definition 5.8]{Jon26}, we define
\[
\PS{\F}{\Lambda} * P \coloneqq 
\setof{
(x, m) \in \PS{\F}{\Lambda} \times P
}{
x \cap \Lambda^m \neq \emptyset
},
\]
and we define $T_{\PS{\F}{\Lambda}} \colon \PS{\F}{\Lambda} * P \to \PS{\F}{\Lambda}$ to map each $(x, m) \in \PS{\F}{\Lambda} * P$ to $\RPA{x}{m} \coloneqq \PSfiltershiftoff{\mu}{x}$, where $\mu$ is the unique element of $x \cap \Lambda^m$.
We can now state the finitely aligned case of \cite[Theorem 5.14]{Jon26}.

\begin{proposition}[Theorem 5.14 of \cite{Jon26}]
Let $(Q, P)$ be a weakly quasi-lattice ordered group, and let $\Lambda$ be a finitely aligned $P$-graph.
Then, $(\PS{\F}{\Lambda}, P, T_{\PS{\F}{\Lambda}})$ is a locally compact and directed monoid action.
\end{proposition}

Recall the semidirect product groupoids of monoid actions from \cref{sec:semidirect_product_groupoid_prelims}.

\begin{definition}[Path groupoid of filters]
\label{def:PGfilter}
Let $(Q, P)$ be a weakly quasi-lattice ordered group, and let $\Lambda$ be a finitely aligned $P$-graph.
We write $\PG{\F}{\Lambda}$ for the semidirect product groupoid of $(\PS{\F}{\Lambda}, P, T_{\PS{\F}{\Lambda}})$.
\end{definition}

\cref{sec:notation} presents $\PG{\F}{\Lambda}$ alongside the other descriptions of the path groupoid of $\Lambda$ considered in this paper.
Since $\PS{\F}{\Lambda}$ has a countable basis of compact open sets, it follows from the discussion in \cref{sec:semidirect_product_groupoid_prelims} that $\PG{\F}{\Lambda}$ is ample, Hausdorff and second-countable. 
If $Q$ is countable and amenable (e.g. if $\Lambda$ is a $k$-graph), then it follows from \cite[Theorem 5.13]{RW17} that $\PG{\F}{\Lambda}$ is amenable (see also \cite[\S 5.4]{Jon26}).
By \cite[Theorem 6.6]{Jon26}, for any finitely aligned $P$-graph $\Lambda$, $\PG{\F}{\Lambda}$ is isomorphic to Spielberg's groupoid from \cite{Spi20}.
By \cite[Remark 6.7]{Jon26}, for any (discrete) row-finite $P$-graph $\Lambda$, $\PG{\F}{\Lambda}$ is isomorphic to the Toeplitz groupoid of \cite[Definition 6.11]{RW17}, so when $\PG{\F}{\Lambda}$ is amenable its C*-algebra coincides with the Toeplitz algebra of 
\cite[Definition 6.11]{RW17}.
We show in \cref{sec:k} that $\PG{\F}{\Lambda}$ is isomorphic to the path groupoids from \cite{FMY05, Yee07} when $\Lambda$ is a $k$-graph.

\subsubsection{Boundary-path groupoid}

While $\PG{\F}{\Lambda}$ from \cref{def:PGfilter} coincides with the \emphasis{Toeplitz} grou-poid of \cite[Definition 6.11]{RW17} for (discrete) row-finite $P$-graphs $\Lambda$, a certain subgroupoid of $\PG{\F}{\Lambda}$ will coincide with the \emphasis{Cuntz--Krieger} groupoid from \cite[Definition 6.18]{RW17} whose C*-algebra is called the Cuntz--Krieger algebra of $\Lambda$.
We refer the reader to \cite{BSV13, arXiv_Hub21, HKLQ24} for C*-algebras generated by representations of $P$-graphs.
In particular, by \cite[Corollary 5.5]{BSV13}, the C*-algebras associated to finitely aligned $P$-graphs in \cite[Theorem 5.3]{BSV13} generalise the Cuntz--Krieger algebras of finitely aligned $k$-graphs from \cite{RSY04}. 
We note that \cite{RW17} does not consider how the Cuntz--Krieger algebra of \cite[Definition 6.18]{RW17} relates to the C*-algebra from \cite[Theorem 5.3]{BSV13}.

Following \cite{BSV13}, we say $x \in \PS{\F}{\Lambda}$ is an \Define{ultrafilter} if $x \subseteq y \in \PS{\F}{\Lambda} \implies x = y$. 
In a directed graph with no \Define{sources} (i.e. vertices that receive no edges), the infinite paths in the graph corresponds with the ultrafilters of the finite-path category of the graph. 
Thus, the set of ultrafilters need not be closed in $\PS{\F}{\Lambda}$. 
For example, in the directed graph with one vertex $v$ and edges $e_n$ for each $n \in \N$, the sequence $(e_ne_n\cdots)_n$ converges to $v$.
Hence, as Paterson puts it in \cite{Pat02}, a ``peculiarity of [the groupoid whose C*-algebra coincides with the Cuntz--Krieger algebra] is that we have to include some \emphasis{finite} paths''. 
Like in the more general setting of finitely aligned categories of paths \cite[Page 729]{Spi12}, this is achieved by including the limit points of ultrafilters of $\Lambda$.

\begin{definition}[Boundary-path space of filters]
\label{def:BPSfilter}
Let $(Q, P)$ be a weakly quasi-lattice ordered group, and let $\Lambda$ be a finitely aligned $P$-graph.
We define $\BPS{\F}{\Lambda}$ to be the closure in $\PS{\F}{\Lambda}$ of the set of ultrafilters of $\Lambda$.
\end{definition}

We say $E \subseteq \Lambda$ is \Define{exhaustive}\label{def:exhaustive} if, for all $\lambda$ with $\range(\lambda) \in \range(E)$, there is a $\mu \in E$ such that $\lambda\Lambda \cap \mu\Lambda \neq \emptyset$.
Given $\mu \in x \in \PS{\F}{\Lambda}$, we say $\mu$ is \Define{extendable}\label{def:extendable} in $x$ if, whenever $E \subseteq \Lambda$ is finite, exhaustive and $\source(\mu) \in \range(E)$, there is a $\nu \in \Lambda$ such that $\mu\nu \in x$.
The following result is a special case of \cite[Theorem 7.8]{Spi14}.
A simplified presentation of Spielberg's proof of the result is given for this special case in \cite[\S 5.4.1]{Jon24}.

\begin{proposition}[Theorem 7.8 of \cite{Spi14}]
A filter $x \in \BPS{\F}{\Lambda}$ if and only if every $\mu \in x$ is extendable in $x$.   
\end{proposition}

\begin{definition}[Boundary-path groupoid of filters]
\label{def:BPGfilter}
Let $(Q, P)$ be a weakly quasi-lattice ordered group, and let $\Lambda$ be a finitely aligned $P$-graph.
We write $\BPG{\F}{\Lambda}$ for the reduction of $\PG{\F}{\Lambda}$ to the closed invariant subset $\BPS{\F}{\Lambda}$ of $\PS{\F}{\Lambda}$.
\end{definition}

Recall from \cref{sec:path_groupoid} that $\PG{\F}{\Lambda}$ is ample, Hausdorff and second-count-able, and if $Q$ is countable and amenable, then $\PG{\F}{\Lambda}$ is amenable. 
Since $\BPG{\F}{\Lambda}$ is the reduction of $\PG{\F}{\Lambda}$ to the closed invariant subset $\BPS{\F}{\Lambda}$ of $\PS{\F}{\Lambda}$, $\BPG{\F}{\Lambda}$ is ample, Hausdorff and second-countable, and if $Q$ is countable and amenable, then $\BPG{\F}{\Lambda}$ is amenable.

\subsection{Graph morphism approach}
\label{sec:morphisms}

Graph morphisms have been used to model paths in $P$-graphs since the inception of higher-rank graphs in \cite{KP00}.
The \emphasis{boundary}-path groupoid of $P$-graphs where $P$ is a finitely generated abelian cancellative monoid was introduced in \cite{CKSS14}.
Describing the \emphasis{path} groupoid of $P$-graphs, which is one of the contributions of this paper, is more technical.
In this section, given any finitely aligned $P$-graph $\Lambda$, where $(Q, P)$ is a weakly quasi-lattice ordered group, we construct a path space $\PS{\M}{\Lambda}$ and a boundary-path space $\BPS{\M}{\Lambda}$ whose elements are certain graph morphisms.
The technicality of this approach emphasises the value of the filter approach, especially in the setting of $P$-graphs.
In \cref{sec:conjugacy_of_filters_and_morphisms}, we show the actions of $P$ on $\PS{\F}{\Lambda}$ and $\PS{\M}{\Lambda}$ are conjugate so that the path and boundary-path groupoids are isomorphic by \cref{thm:conjugate_actions_have_isomorphic_groupoids}.
In \cref{sec:k}, we show the path and boundary-path groupoids coincide with those of \cite{Yee07} when $\Lambda$ is a finitely aligned (discrete) 
$k$-graph.

\subsubsection{Path groupoid}

The first challenge is to define a sensible notion of $P$-path prototypes that generalise the rank-$k$ path prototypes $\Omega_{k, m}$ of \cite{Yee07}. 

\begin{proposition}
\label{prop:union_of_nested_graphs}
Let $(Q, P)$ be a weakly quasi-lattice ordered group, and let $(\Lambda_n)$ be a sequence of $P$-graphs such that, for all $n \in \N$, $\Lambda_n \subseteq \Lambda_{n+1}$ and the inclusion map from $\Lambda_n$ to $\Lambda_{n+1}$ is a $P$-graph morphism.
The set $\Lambda \coloneqq \bigcup_{n} \Lambda_n$ is a $P$-graph with operations $(\lambda, \mu) \mapsto \lambda\mu$ from $\composablepairs{\Lambda} \coloneqq \bigcup_{n} \composablepairs{\Lambda_n}$ to $\Lambda$ and $\degree(\lambda) \coloneqq \degree_n(\lambda)$, where $\lambda \in \Lambda_n$. 
We call $\Lambda$ the \Define{direct limit} of $(\Lambda_n)$.
\end{proposition}
\begin{proof}
We have that $\Lambda$ is a small category since the inclusion maps are functors.
The map $\degree \colon \Lambda \to P$ is well-defined because the inclusion maps are degree-preserving.
Suppose $\degree(\lambda) = mn$ for some $m, n \in P$.
Say $\lambda \in \Lambda_k$ so that $\degree_k(\lambda) = mn$.
Because $\Lambda_k$ is a $P$-graph, there are unique $\mu \in \Lambda_k^m \subseteq \Lambda^m$ and $\nu \in \Lambda_k^n \subseteq \Lambda^n$ such that $\lambda = \mu\nu$.
Suppose $\mu' \in \Lambda^m$ and $\nu' \in \Lambda^n$ such that $\lambda = \mu'\nu'$.
Find $i$ and $j$ such that $\mu' \in \Lambda_i$ and $\nu' \in \Lambda_j$.
Then, $\mu' \in \Lambda_i^m$ and $\nu' \in \Lambda_j^n$ satisfy $\lambda = \mu'\nu'$.
Let $l = \max\{i, j, k\}$ so that $\mu, \nu, \mu', \nu', \lambda \in \Lambda_l$.
Because the inclusion maps are degree-preserving, $\mu' \in \Lambda_l^m$ and $\nu' \in \Lambda_l^n$, which satisfy $\lambda = \mu'\nu'$. 
Thus, the unique factorisation property of $\degree_l$ in $\Lambda_l$ implies $\mu' = \mu$ and $\nu' = \nu$, as required.
\end{proof}

Recall $\Omega_{P, m}$ from \cref{ex:path_prototypes_finite}.
We say a sequence $(m_n) \subseteq P$ is \Define{$\leq$-increasing} if $m_j \leq m_k$ whenever $j < k$. 
In \cref{prop:increasing_sequences_in_P_yield_nested_P_graphs}, we show each $\leq$-increasing sequence $(m_n)$ gives rise to a direct limit $\Omega_{P, (m_n)}$, the collection of which coincides with the collection of $\Omega_{k, m}$ for $m \in (\N \cup \{\infty\})^k$ in the sense of \cite{Yee07} when $(Q, P) = (\Z^k, \N^k)$ (see \cref{sec:k}).

\begin{proposition}
\label{prop:increasing_sequences_in_P_yield_nested_P_graphs}
Let $(Q, P)$ be a weakly quasi-lattice ordered group.
Suppose $(m_n) \subseteq P$ is $\leq$-increasing.
For all $n \in \N$, $\Omega_{P, m_n} \subseteq \Omega_{P, m_{n+1}}$ and the inclusion map from $\Omega_{P, m_n}$ to $\Omega_{P, m_{n+1}}$ is a $P$-graph morphism, and we denote the direct limit of $(\Omega_{P, m_n})$ by $\Omega_{P, (m_n)}$.
\end{proposition}
\begin{proof}
For all $n \in \N$, $\Omega_{P, m_n} \subseteq \Omega_{P, m_{n+1}}$ holds because $\leq$ is transitive.
The inclusion maps from $\Omega_{P, m_n}$ to $\Omega_{P, m_{n+1}}$ are $P$-graph morphisms because the composable pairs, the composition map, the resulting units and the degree maps in $P$-graphs of the form $\Omega_{P, m}$ are all independent of the parameter $m$.
Thus, the direct limit exists by \cref{prop:union_of_nested_graphs}.
\end{proof}

\begin{definition}[Path space of graph morphisms]
\label{def:PSmorph}
Let $(Q, P)$ be a weakly quasi-lattice ordered group, and let $\Lambda$ 
be a finitely aligned $P$-graph. 
We define $\PS{\M}{\Lambda}$ to be the set of graph morphisms $x \colon \Omega_{P, (m_n)} \to \Lambda$, where $(m_n) \subseteq P$ is $\leq$-increasing, and we write $\morphdom{x}$ for the domain of $x$ so that $x \colon \morphdom{x} \to \Lambda$.
\end{definition}    

\begin{remark}
\label{rem:extending_degree_map}
When $\Lambda$ is a $k$-graph, the degree map $\degree$ on $\Lambda$ can be extended to $\PS{\ofYeend}{\Lambda}$ by letting $\degree(x) = m \in (\N \cup \{\infty\})^k$, where $x \colon \Omega_{k, m} \to \Lambda$.
In our generalisation, we can have $\Omega_{P, (l_n)} = \Omega_{P, (m_n)}$ even if $(l_n) \neq (m_n)$, so it would not be well-defined to let $\degree(x) = (m_n)$, where $x \colon \Omega_{P, (m_n)} \to \Lambda$.
Instead, we extend $\degree$ to $\PS{\M}{\Lambda}$ as follows.
Given $\leq$-increasing sequences $(l_n), (m_n) \subseteq P$, define $(l_n) \prec (m_n)$ if and only if, for all $j \in \N$, there is a $K \in \N$ such that $l_j \leq m_K$ (in which case $l_j \leq m_k$ for all $k \geq K$ because $(m_n) \subseteq P$ is $\leq$-increasing).
Because the relation $\leq$ on $P$ is transitive, so is $\prec$.
The relation $\prec$ is also reflexive (and hence a preorder), but $\prec$ need not be antisymmetric (e.g. $(1, 3, 5, \dots)$ and $(2, 4, 6, \dots)$ in $\N$).
Because $\prec$ is a preorder, the relation $(l_n) \sim (m_n) \iff (l_n) \prec (m_n) \prec (l_n)$ is an equivalence relation on the collection of $\leq$-increasing sequences in $P$.
For any $\leq$-increasing sequences $(l_n), (m_n) \subseteq P$, $(l_n) \sim (m_n) \iff \Omega_{P, (l_n)} = \Omega_{P, (m_n)}$.
Therefore, we extend $\degree$ to $\PS{\M}{\Lambda}$ by defining $\degree(x) = [(m_n)]_\sim$, where $x \colon \Omega_{P, (m_n)} \to \Lambda$.
\end{remark}

Now we endow $\PS{\M}{\Lambda}$ with a topology. 
We write $e\morphdom{x} \coloneqq \setof{(p, q) \in \morphdom{x}}{p = e}$.
Given finite (possibly empty) $K_1, K_2 \subseteq \Lambda$ and $\mu \in \Lambda$, let
\begin{align*}
\Fell{\PS{\M}{\Lambda}}{K_1}{K_2}
&\coloneqq \setof{
x \in \PS{\M}{\Lambda}
}{
K_1 \subseteq x(e\morphdom{x}) \subseteq \Lambda \setminus K_2
}, \\
\Fell{\PS{\M}{\Lambda}}{\mu}{K_2}
&\coloneqq \Fell{\PS{\M}{\Lambda}}{\{\mu\}}{K_2}, \text{ and } \\
\Fellin{\PS{\M}{\Lambda}}{\mu}
&\coloneqq \Fell{\PS{\M}{\Lambda}}{\mu}{\emptyset}.
\end{align*}
The collection of such $\Fell{\PS{\M}{\Lambda}}{K_1}{K_2}$ is a basis for a topology on $\PS{\M}{\Lambda}$.
\label{basis_for_PSmorph}
Define
\[
\PS{\M}{\Lambda} * P \coloneqq 
\setof{
(x, m) \in \PS{\M}{\Lambda} \times P
}{
(e, m) \in \morphdom{x}
}.
\]
In light of \cref{rem:extending_degree_map}, it is worth demonstrating how 
containment $(x, p) \in \PS{\M}{\Lambda} * P$ is independent of the presentation of $\morphdom{x}$.

\begin{lemma}
\label{lem:actionable_morphisms}
Let $(Q, P)$ be a weakly quasi-lattice ordered group, and let $\Lambda$ be a finitely aligned $P$-graph. 
The following are equivalent:
\begin{enumerate}
\item \label{actionable_1} $(x, p) \in \PS{\M}{\Lambda} * P$;
\item \label{actionable_2} for all $\leq$-increasing $(m_n)$, $\morphdom{x} = \Omega_{P, (m_n)}$ implies $p \leq m_n$ eventually;
\item \label{actionable_3} there is some $\leq$-increasing $(m_n)$ with $\morphdom{x} = \Omega_{P, (m_n)}$ and $p \leq m_n$ eventually;
\item \label{actionable_4} there is some $\leq$-increasing $(m_n)$ with $\morphdom{x} = \Omega_{P, (m_n)}$ and $p \leq m_n$ for all $n$.
\end{enumerate}
\end{lemma}
\begin{proof}
Suppose \eqref{actionable_1} holds, so $(e, p) \in \morphdom{x}$, and suppose $(m_n)$ satisfies $\morphdom{x} = \Omega_{P, (m_n)}$.
If $(m_n)$ is $\leq$-increasing, $p \leq m_n$ eventually, hence \eqref{actionable_2} holds. 
For every $x \in \PS{\M}{\Lambda}$, there is a $\leq$-increasing $(m_n)$ satisfying $\morphdom{x} = \Omega_{P, (m_n)}$, so \eqref{actionable_2} implies \eqref{actionable_3}. 
Suppose \eqref{actionable_3} holds, so there is some $(m_n)$ and $N \in \N$ such that $\morphdom{x} = \Omega_{P, (m_n)}$ and $p \leq m_n$ for all $n \geq N$.
Then, $(m_n)_{n \geq N}$ is satisfactory.
If \eqref{actionable_4} holds, then $(e, p) \in \Omega_{P, m_1} \subseteq \morphdom{x}$, and so $(x, p) \in \PS{\M}{\Lambda} * P$. 
That is, \eqref{actionable_1} holds.
\end{proof}

\begin{lemma}
\label{lem:action_on_PS_of_morphisms}
Let $(Q, P)$ be a weakly quasi-lattice ordered group, and let $\Lambda$ be a finitely aligned $P$-graph. 
For each $(x, m) \in \PS{\M}{\Lambda} * P$, $\morphdom{x} = \Omega_{P, (m_n)}$ for some $\leq$-increasing $(m_n) \subseteq P$ with $m \leq m_n$ for all $n$.
The map
\[
\RPA{x}{m} \colon \Omega_{P, (m^{-1}m_n)} \to \Lambda, \quad 
(p, q) \mapsto (\RPA{x}{m})(p, q) \coloneqq x(mp, mq),
\]
is a graph morphism, and $(p, q) \in \morphdom{\RPA{x}{m}} \iff (mp, mq) \in \morphdom{x}$.
\end{lemma}
\begin{proof}
By \cref{lem:actionable_morphisms}, $\morphdom{x} = \Omega_{P, (m_n)}$ for some $\leq$-increasing $(m_n) \subseteq P$ with $m \leq m_n$ for all $n$.
By the definition of $\leq$ and because $\leq$ is left invariant, $(m^{-1}m_n)$ is a $\leq$-increasing subset of $P$, and hence has direct limit $\Omega_{P, (m^{-1}m_n)}$ as per \cref{prop:increasing_sequences_in_P_yield_nested_P_graphs}. 
Because $m \leq m_n$ for all $n$, $(p, q) \in \Omega_{P, (m^{-1}m_n)} \iff (mp, mq) \in \Omega_{P, (m_n)}$, which implies $\RPA{x}{m}$ is well-defined.

To see $\RPA{x}{m}$ is a functor, suppose $p \leq q \leq r \leq m^{-1}m_n$ for some $n$ (i.e. $((p, q), (q, r))$ is a composable pair in $\Omega_{P, (m^{-1}m_n)}$). 
By left invariance, 
\[
mp \leq mq \leq mr \leq m_n,
\] 
so $((mp, mq), (mq, mr)) \in \composablepairs{\Omega_{P, (m_n)}}$. 
Since $x$ is a functor, 
\begin{align*}
((\RPA{x}{m})(p, q), (\RPA{x}{m})(q, r)) 
&= (x(mp, mq), x(mq, mr)) \in \composablepairs{\Lambda} 
\text{ and } \\
(\RPA{x}{m})(p, q)(\RPA{x}{m})(q, r) 
&= x(mp, mq)x(mq, mr) \\ 
&= x((mp, mq)(mq, mr)) \\ 
&= x(mp, mr) = (\RPA{x}{m})(p, r) \\
&= (\RPA{x}{m})((p, q)(q, r)),
\end{align*} 
so $\RPA{x}{m}$ is a functor.
Also, 
\[
\degree((\RPA{x}{m})(p, q)) 
= \degree(x(mp, mq)) 
= \degree(mp, mq) 
= (mp)^{-1}mq 
= p^{-1}q 
= \degree(p, q),
\] 
so $\RPA{x}{m}$ is degree-preserving. 
That is, $\RPA{x}{m}$ is a graph morphism.
\end{proof}

We define $T_{\PS{\M}{\Lambda}} \colon \PS{\M}{\Lambda} * P \to \PS{\M}{\Lambda}$ such that each $(x, m) \in \PS{\M}{\Lambda} * P$ maps to $\RPA{x}{m}$ as in \cref{lem:action_on_PS_of_morphisms}.

\begin{proposition}
Let $(Q, P)$ be a weakly quasi-lattice ordered group, and let $\Lambda$ be a finitely aligned $P$-graph.
Then, $(\PS{\M}{\Lambda}, P, T_{\PS{\M}{\Lambda}})$ is a monoid action.
\end{proposition}
\begin{proof}
We check the axioms \ref{S1} and \ref{S2} hold.
For all $x \in \PS{\M}{\Lambda}$, $(e, e) \in \morphdom{x}$, so $(x, e) \in \PS{\M}{\Lambda} * P$, and $(\RPA{x}{e})(p, q) = x(ep, eq) = x(p, q)$, for all $(p, q) \in \morphdom{x}$. 
That is, $\RPA{x}{e} = x$, so \ref{S1} holds.
Suppose $(x, pq) \in \PS{\M}{\Lambda} * P$. 
By \cref{lem:actionable_morphisms}, there is some $\leq$-increasing $(m_n)$ satisfying $\morphdom{x} = \Omega_{P, (m_n)}$ and $pq \leq m_n$ for all $n$.
Since $\leq$ is transitive, $p \leq m_n$ for all $n$, so $(x, p) \in \PS{\M}{\Lambda} * P$.
By left invariance of $\leq$, $q \leq p^{-1}m_n$ for all $n$, and $\morphdom{\RPA{x}{p}} = \Omega_{P, (p^{-1}m_n)}$, so $(\RPA{x}{p}, q) \in \PS{\M}{\Lambda} * P$.
Now suppose $(x, p), (\RPA{x}{p}, q) \in \PS{\M}{\Lambda} * P$.
We need to show $(x, pq) \in \PS{\M}{\Lambda} * P$.
By \cref{lem:actionable_morphisms}, there is some $\leq$-increasing $(m_n)$ satisfying $\morphdom{x} = \Omega_{P, (m_n)}$ and $p \leq m_n$ for all $n$.
Then, $\morphdom{\RPA{x}{p}} = \Omega_{P, (p^{-1}m_n)}$.
Since $(\RPA{x}{p}, q) \in \PS{\M}{\Lambda} * P$, we have that $(e, q) \in \Omega_{P, (p^{-1}m_n)}$, so $q \leq p^{-1}m_n$ for some $n$. 
By left invariance of $\leq$, we have $pq \leq m_n$, and so $(e, pq) \in \Omega_{P, m_n} \subseteq \morphdom{x}$. 
That is, $(x, pq) \in \PS{\M}{\Lambda} * P$.
In the above setting, we need to show $\RPA{x}{(pq)} = \RPA{(\RPA{x}{p})}{q}$.
Recall from \cref{lem:action_on_PS_of_morphisms} that $(mp, mq) \in \morphdom{x} \iff (p, q) \in \morphdom{\RPA{x}{m}}$, so 
\begin{align*}
(j, k)& \in \morphdom{\RPA{x}{(pq)}} \\
\iff &(pqj, pqk) \in \morphdom{x} \quad \text{(because $(x, pq) \in 
\PS{\M}{\Lambda} * P$)} \\
\iff &(qj, qk) \in \morphdom{\RPA{x}{p}} \quad \text{(because $(x, p) 
\in \PS{\M}{\Lambda} * P$)} \\
\iff &(j, k) \in \morphdom{\RPA{(\RPA{x}{p})}{q}} \quad 
\text{(because $(\RPA{x}{p}, q) \in \PS{\M}{\Lambda} * P$)},
\end{align*}
so $\morphdom{\RPA{x}{(pq)}} = \morphdom{\RPA{(\RPA{x}{p})}{q}}$. 
Moreover, for any $(j, k) \in \morphdom{\RPA{x}{(pq)}} = \morphdom{\RPA{(\RPA{x}{p})}{q}}$, 
\[
(\RPA{x}{(pq)})(j, k)
= x(pqj, pqk)
= (\RPA{x}{p})(qj, qk)
= (\RPA{(\RPA{x}{p})}{q})(j, k).
\]
Therefore, \ref{S2} holds.
\end{proof}

We will prove in \cref{thm:conjugacy} that the action on $\PS{\M}{\Lambda}$ is locally compact and directed, hence admits the following semidirect product groupoid $\PG{\M}{\Lambda}$, which coincides with $\PG{\ofYeend}{\Lambda}$ when $\Lambda$ is a finitely aligned discrete $k$-graph (see \cref{sec:k}).

\begin{definition}[Path groupoid of graph morphisms]
\label{def:PGmorph}
Let $(Q, P)$ be a weakly quasi-lattice ordered group, and let $\Lambda$ be a finitely aligned $P$-graph.
We write $\PG{\M}{\Lambda}$ for the semidirect product groupoid of $(\PS{\M}{\Lambda}, P, T_{\PS{\M}{\Lambda}})$.
\end{definition}

\subsubsection{Boundary-path groupoid}

We use exhaustive subsets of $\Lambda$ (see \cref{def:exhaustive}) to define a boundary-path space of $\Lambda$ consisting of graph morphisms.

\begin{definition}[Boundary-path space of graph morphisms]
\label{def:BPSmorph}
Let $(Q, P)$ be a weakly quasi-lattice ordered group, and let $\Lambda$ be a finitely aligned $P$-graph. 
We define $\BPS{\M}{\Lambda}$ to be set of $x \in \PS{\M}{\Lambda}$ such that, if $(e, m) \in \morphdom{x}$ and $E \subseteq \Lambda$ is finite, exhaustive and $\source(x(e, m)) \in \range(E)$, then there is a $\nu \in E$ such that $x(m, m\degree(\nu)) = \nu$.
\end{definition}

We will prove in \cref{cor:isom} that $\BPS{\M}{\Lambda}$ is closed and invariant so that it makes sense to define the following reduction of $\PG{\M}{\Lambda}$.

\begin{definition}[Boundary-path groupoid of graph morphisms]
\label{def:BPGmorph}
Let $(Q, P)$ be a weakly quasi-lattice ordered group, and let $\Lambda$ be a finitely aligned $P$-graph.
We write $\BPG{\M}{\Lambda}$ for the reduction of $\PG{\M}{\Lambda}$ to the closed invariant subset $\BPS{\M}{\Lambda}$ of $\PS{\M}{\Lambda}$.
\end{definition}

\section{Conjugacy of filters and graph morphisms}    
\label{sec:conjugacy_of_filters_and_morphisms}

It was observed in \cite[Remarks 2.2]{KP00} that each graph morphism $x$ from $\Omega_{k}$ to a row-finite $k$-graph with no sources is completely determined by the set $\setof{x(0, m)}{m \in \N^k}$. 
Given any finitely aligned $P$-graph $\Lambda$, we show in \cref{thm:conjugacy} that this association between graph morphisms and subsets of $\Lambda$ forms a conjugacy between the monoid actions on $\PS{\M}{\Lambda}$ and $\PS{\F}{\Lambda}$ from the previous section. 
We note that for finitely aligned $k$-graphs, graph morphisms have been related to certain $\preceq$-increasing sequences, called $N$-paths, in \cite[Page 2763]{PW05}; and for semigroupoids, graph morphism-like objects have been related to filter-like objects in the semigroupoid \cite[Proposition 19.11]{Exe08}. 
However, as far as we are aware, it has not been made explicit how such identifications give rise to isomorphic groupoids, nor has this been generalised to the setting of finitely aligned $P$-graphs. 
We first need a lemma that shows all filters are unions of principal filters. 

\begin{lemma}
\label{lem:filters_are_unions_of_principal_filters}
Let $(Q, P)$ be a weakly quasi-lattice ordered group, and let $\Lambda$ be a finitely aligned $P$-graph.
For any $y \in \PS{\F}{\Lambda}$, there is a $\preceq$-increasing $(\mu_n) \subseteq y$ such that $y = \bigcup_{n}\principal{\mu_n}$. 
The sequence $(\degree(\mu_n))$ is $\leq$-increasing, and the direct limit $\Omega_{P, (\degree(\mu_n))}$ is independent of $(\mu_n)$.
\end{lemma}
\begin{proof}
Since $\Lambda$ is countable, we can enumerate the elements of $y$ and  write $y = \{\mu_1', \mu_2', \dots\}$. 
We build a sequence $(\mu_n)$ inductively. 
Let $\mu_1 \coloneqq \mu_1'$. 
Suppose for some $k \geq 1$ there is a $\mu_k \in y$ such that $j \leq k$ implies $\mu_j' \preceq \mu_k$. 
Then, $\mu_{k+1}', \mu_k \in y$, and $y$ is directed, so we can define  $\mu_{k+1} \in y$ such that $\mu_{k+1}', \mu_k \preceq \mu_{k+1}$. 
By transitivity of $\preceq$, $j \leq k+1$ implies $\mu_j' \preceq \mu_{k+1}$. 
Therefore, by induction, there is a $\preceq$-increasing $(\mu_n) \subseteq y$ such that $j \leq n$ implies $\mu_j' \preceq \mu_n$ and $(\mu_n)$ is $\preceq$-increasing. 
Now fix $\mu \in y$, so $\mu = \mu_j' \preceq \mu_j$ for some $j$. 
Thus, $y \subseteq \bigcup_{n}\principal{\mu_n}$. 
Because $y$ is hereditary, we also have $\bigcup_{n}\principal{\mu_n} \subseteq y$.

Because $\degree$ is a functor, for any $\preceq$-increasing $(\mu_n) \subseteq \Lambda$, $(\degree(\mu_n))$ is $\leq$-increasing and hence has a direct limit $\Omega_{P, (\degree(\mu_n))}$. 
Suppose $(\mu_n), (\mu_n') \subseteq y$ are both $\preceq$-increasing such that $y = \bigcup_{n}\principal{\mu_n} = \bigcup_{n} \principal{\mu_n'}$.
We claim $\Omega_{P, (\degree(\mu_n))} = \Omega_{P, (\degree(\mu_n'))}$. 
To prove the claim, recall from \cref{rem:extending_degree_map} that it is equivalent to show $(\degree(\mu_n)) \prec (\degree(\mu_n')) \prec (\degree(\mu_n))$.
For each $j \in \N$, $\mu_j \in \principal{\mu_j} \subseteq \bigcup_{n} \principal{\mu_n'}$, so there is a $K \in \N$ such that $\mu_j \preceq \mu_K'$.
Because $\degree$ is a functor, $\degree(\mu_j) \leq \degree(\mu_K')$.
Hence, $(\degree(\mu_n)) \prec (\degree(\mu_n'))$.
A symmetric argument shows $(\degree(\mu_n')) \prec (\degree(\mu_n))$, and so $\Omega_{P, (\degree(\mu_n))} = \Omega_{P, (\degree(\mu_n'))}$, as required.
\end{proof}

\begin{theorem}
\label{thm:conjugacy}
Let $(Q, P)$ be a weakly quasi-lattice ordered group, and let $\Lambda$ be a finitely aligned $P$-graph.
The map 
\[
h \colon \PS{\M}{\Lambda} \to \PS{\F}{\Lambda}, \quad 
x \mapsto h(x) \coloneqq x(e\morphdom{x}),
\]
is a conjugacy. 
\end{theorem}
\begin{proof}
We first show $h(x) \in \PS{\F}{\Lambda}$.
Observe $(e, e) \in e\morphdom{x}$, so $h(x) \neq \emptyset$.
Fix $x(e, p), x(e, q) \in h(x)$, where $(e, p), (e, q) \in \morphdom{x}$.
By definition of $\PS{\M}{\Lambda} * P$, we have $(x, p), (x, q) \in \PS{\M}{\Lambda} * P$, and the characterisation in \cref{lem:actionable_morphisms} of $\PS{\M}{\Lambda} * P$ implies $p, q \leq m_n$ eventually, where $\morphdom{x} = \Omega_{P, (m_n)}$. 
Thus, there is an $n \in \N$ such that $p, q \leq m_n$. 
Because $x$ is a functor, $x(e, p), x(e, q) \preceq x(e, m_n)$, so $h(x)$ is directed. 
Fix $x(e, r) \in h(x)$, and suppose $\mu\nu = x(e, r)$. 
Using that $x$ is a functor and the unique factorisation property, it is straightforward to check that $\mu = x(e, \degree(\mu)) \in h(x)$, so $h(x)$ is hereditary.
Now we show \ref{C1} holds, i.e. $h$ is a homeomorphism.
Suppose $h(x) = h(y)$.
Then, $x(e\morphdom{x}) = y(e\morphdom{y})$.
We show $\morphdom{x} = \morphdom{y}$.
If $(j, k) \in \morphdom{x}$, $(e, k) \in e\morphdom{x}$, so $x(e, k) \in y(e\morphdom{y})$. 
There is a $k' \in P$ such that $x(e, k) = y(e, k')$, so $k = k'$ because $x$ and $y$ are degree-preserving. 
Hence, $(e, k) \in \morphdom{y}$, and so $(j, k) \in \morphdom{y}$. 
A symmetric argument shows the other inclusion.
Because $x$ is a functor, $x(e, j)x(j, k) = y(e, j)y(j, k)$, and the unique factorisation property implies $x(j, k) = y(j, k)$ because $x$ is degree-preserving. 
Thus, $x = y$, and so $h$ is injective.
Let $y \in \PS{\F}{\Lambda}$.
By \cref{lem:filters_are_unions_of_principal_filters}, there is a $\preceq$-increasing $(\mu_n) \subseteq y$ such that $y = \bigcup_{n} \principal{\mu_n}$ and $(\degree(\mu_n))$ is $\leq$-increasing so that the direct limit $\Omega_{P, (\degree(\mu_n))}$ exists.
Fix $(j, k) \in \Omega_{P, (\degree(\mu_n))}$, so $k \leq \degree(\mu_n)$ for some $n$.
By the unique factorisation property, there are unique $\kappa \in \Lambda^k$ and $\lambda \in \Lambda^{k^{-1}\degree(\mu_n)}$ such that $\kappa\lambda = \mu_n \in y$. 
Because $y$ is hereditary, $\kappa \in y$.
Notice that, if $(\mu_n')$ is another sequence as in \cref{lem:filters_are_unions_of_principal_filters} so that $\Omega_{P, (\degree(\mu_n))} = \Omega_{P, (\degree(\mu_n'))}$, then the above argument yields a $\kappa' \in \Lambda^k \cap y$ and so $\kappa = \kappa'$ by the injectivity of $\degree|_y$. 
That is, $\kappa$ does not depend on the choice of $(\mu_n)$. 
The unique factorisation property implies there are unique $\theta \in \Lambda^j$ and $\iota \in \Lambda^{j^{-1}k}$ such that $\kappa = \theta\iota$. 
In this way, $(j, k)$ uniquely determines $\iota \in \Lambda$, so we can define
\[
x \colon \Omega_{P, (\degree(\mu_n))} \to \Lambda, \quad 
(j, k) \mapsto x(j, k) \coloneqq \iota.
\]
We show $x(e\morphdom{x}) = y$.
Fix $(e, k) \in e\morphdom{x}$.
In the above setting, where $x(e, k)$ is defined, we have $\theta \in \Lambda^e$, so $\theta = \range(\iota)$. 
Thus, $x(e, k) = \iota = \range(\iota)\iota = \kappa \in y$. 
Now fix $\kappa' \in y$. 
Then, $\kappa' \preceq \mu_n$ for some $n$, which implies $\degree(\kappa') \leq \degree(\mu_n)$, and so $(e, \degree(\kappa')) \in e\Omega_{P, (\degree(\mu_n))} = e\morphdom{x}$. 
We need to show $x(e, \degree(\kappa')) = \kappa'$. 
By definition of $x(e, \degree(\kappa'))$, $\kappa \in \Lambda^{\degree(\kappa')}$, so $\degree(\kappa) = \degree(\kappa')$ and $\kappa, \kappa' \in y$. 
By the injectivity of $\degree|_y$, we have that $x(e, \degree(\kappa')) = \kappa = \kappa'$, as required.
Therefore, $x(e\morphdom{x}) = y$.
We show $x$ is a graph morphism.
Fix $((i, j), (j, k)) \in \composablepairs{\Omega_{P, (\degree(\mu_n))}}$.
As per the definition of $x$, there are unique $\kappa_{j, k} \in y \cap \Lambda^k$, $\theta_{j, k} \in y \cap \Lambda^j$ and $\iota_{j, k} \in \Lambda^{j^{-1}k}$ such that $\kappa_{j, k} = \theta_{j, k}\iota_{j, k}$.
Similarly, there are unique $\kappa_{i, j} \in y \cap \Lambda^j$, $\theta_{i, j} \in y \cap \Lambda^i$ and $\iota_{i, j} \in \Lambda^{i^{-1}j}$ such that $\kappa_{i, j} = \theta_{i, j}\iota_{i, j}$. 
By the injectivity of $\degree_y$, $\theta_{j, k} = \kappa_{i, j}$. 
Hence, $\theta_{j, k} = \theta_{i, j}\iota_{i, j}$, so $\source(\iota_{i, j}) = \source(\theta_{j, k}) = \range(\iota_{j, k})$. 
That is, $(x(i, j), x(j, k)) \in \composablepairs{\Lambda}$.
We need to show $x(i, k) = x(i, j)x(j, k)$. 
Notice $\theta_{i, k}, \theta_{i, j} \in y \cap \Lambda^i$, so $\theta_{i, k} = \theta_{i, j}$ by the injectivity of $\degree|_y$. 
Because of left cancellation, it suffices to show $\theta_{i, k}\iota_{i, k} = \theta_{i, j}\iota_{i, j}\iota_{j, k}$, equivalently $\kappa_{i, k} = \kappa_{i, j}\iota_{j, k}$. 
Observe 
\[
\degree(\kappa_{i, k}) 
= k 
= j(j^{-1}k) 
= \degree(\kappa_{i, j})\degree(\iota_{j, k}) 
= \degree(\kappa_{i, j}\iota_{j, k}),
\]
and so it is enough to show $\kappa_{i, k}, \kappa_{i, j}\iota_{j, k} \in y$ because of the injectivity of $\degree_y$. 
We have $\kappa_{i, k} \in y$ by construction of $\kappa_{i, k}$. 
Since $\theta_{j, k} = \theta_{i, j}\iota_{i, j}$, we have
\[
\kappa_{i, j}\iota_{j, k}
= \theta_{i, j}\iota_{i, j}\iota_{j, k} 
= \theta_{j, k}\iota_{j, k} 
= \kappa_{j, k}
\in y,
\]
so both $\kappa_{i, k}, \kappa_{i, j}\iota_{j, k} \in y$, as required. 
That is, $x(i, k) = x(i, j)x(j, k)$, and so $x$ is a functor.
Moreover, $x$ is degree-preserving by construction since $\iota_{j, k} \in \Lambda^{j^{-1}k}$. 
Therefore, $x \in \PS{\M}{\Lambda}$, and so $h$ is surjective.

Recall the basis for the topology on $\PS{\M}{\Lambda}$ from \cref{basis_for_PSmorph}.
Notice the bijection $h$ identifies each basic open set $\Fell{\PS{\F}{\Lambda}}{K_1}{K_2}$ of $\PS{\F}{\Lambda}$ with $\Fell{\PS{\M}{\Lambda}}{K_1}{K_2}$. 
Thus, $h$ is a homeomorphism.

Now we show \ref{C2} holds, i.e. for each $m \in P$, $h(\RPAdom{\PS{\M}{\Lambda}}{m}) = \RPAdom{\PS{\F}{\Lambda}}{m}$.
Fix $x \in \RPAdom{\PS{\M}{\Lambda}}{m}$. 
By definition of the action of $P$ on $\PS{\M}{\Lambda}$, $(e, m) \in e\morphdom{x}$, and so $x(e, m) \in h(x)$. 
Because $x$ is a graph morphism, $\degree(x(e, m)) = m$. 
Hence, $h(x) \in \RPAdom{\PS{\F}{\Lambda}}{m}$.
For the other inclusion, if $h(x) \in \RPAdom{\PS{\F}{\Lambda}}{m}$, then there is a $\mu \in h(x) \cap \Lambda^m$. 
By definition of $h$, $\mu = x(e, m')$ for some $(e, m') \in e\morphdom{x}$. 
Notice $m = \degree(\mu) = \degree(x(e, m')) = m'$ because $x$ is degree-preserving.
Thus, $(e, m) \in e\morphdom{x}$, which means $x \in \RPAdom{\PS{\M}{\Lambda}}{m}$, so $h(x) \in h(\RPAdom{\PS{\M}{\Lambda}}{m})$.

Lastly, we show \ref{C3} holds, i.e. $h(\RPA{x}{m}) = \RPA{h(x)}{m}$, for all $m \in P$ and $x \in \RPAdom{\PS{\M}{\Lambda}}{m}$.
By \ref{C2}, $h(x) \in \RPAdom{\PS{\F}{\Lambda}}{m}$, so there is a $\mu \in h(x) \cap \Lambda^m$ and $\RPA{h(x)}{m} = \PSfiltershiftoff{\mu}{h(x)}$. 
Hence, we need to show $h(\RPA{x}{m}) = \PSfiltershiftoff{\mu}{h(x)}$. 
Fix $\nu \in h(\RPA{x}{m}) = (\RPA{x}{m})(e\morphdom{\RPA{x}{m}})$. 
Then, $\nu = (\RPA{x}{m})(e, n) = x(m, mn)$ for some $(e, n) \in e\morphdom{\RPA{x}{m}}$. 
Observe $x(e, m) \in h(x) \cap \Lambda^m$ because $x$ is degree-preserving, so $x(e, m) = \mu$ by the injectivity of $\degree|_{h(x)}$. 
Now 
\[
x(e, mn) 
= x(e, m)x(m, mn) 
= \mu\nu,
\] 
so $\mu\nu \in x(e\morphdom{x}) = h(x)$. 
Thus, $\nu \in \PSfiltershiftoff{\mu}{h(x)}$.
Now fix $\nu \in \PSfiltershiftoff{\mu}{h(x)}$, so $\mu\nu \in h(x) = x(e\morphdom{x})$. 
Then, $\mu\nu = x(e, l)$ for some $(e, l) \in e\morphdom{x}$. 
Compute $l = \degree(x(e, l)) = \degree(\mu)\degree(\nu)$, so 
\[
\mu\nu 
= x(e, l) 
= x(e, \degree(\mu))x(\degree(\mu), \degree(\mu)\degree(\nu)).
\] 
By the unique factorisation property, 
\[
\nu 
= x(\degree(\mu), \degree(\mu)\degree(\nu)) 
= (\RPA{x}{\degree(\mu)})(e, \degree(\nu)) 
= (\RPA{x}{m})(e, \degree(\nu)) 
\in h(\RPA{x}{m}).
\]
Therefore, $h$ is a conjugacy.
\end{proof}

\begin{corollary}
Let $(Q, P)$ be a weakly quasi-lattice ordered group, and let $\Lambda$ be a finitely aligned $P$-graph.
The action $(\PS{\M}{\Lambda}, P, T_{\PS{\M}{\Lambda}})$ is locally compact and directed.
\end{corollary}
\begin{proof}
Recall from \cref{prop:conjugacy_preserves_properties} that conjugacies preserve local compactness and directedness.
Since the action $(\PS{\F}{\Lambda}, P, T_{\PS{\F}{\Lambda}})$ is locally compact and directed, we now have that 
$(\PS{\M}{\Lambda}, P, T_{\PS{\M}{\Lambda}})$ is locally compact and directed (and hence $\PG{\M}{\Lambda}$ from \cref{def:PGmorph} is well-defined).
\end{proof}

\begin{corollary}
\label{cor:isom}
Let $(Q, P)$ be a weakly quasi-lattice ordered group, and let $\Lambda$ be a finitely aligned $P$-graph.
The map 
\[
\psi_h \colon \PG{\M}{\Lambda} \to \PG{\F}{\Lambda}, \quad 
(x, q, y) \mapsto (x(e\morphdom{x}), q, y(e\morphdom{y})),
\]
is an isomorphism. 
The set $\BPS{\M}{\Lambda}$ is closed and invariant, and $\psi_h$ restricts to an isomorphism from $\BPG{\M}{\Lambda} = \reduction{\PG{\M}{\Lambda}}{\BPS{\M}{\Lambda}}$ to $\BPG{\F}{\Lambda} = \reduction{\PG{\F}{\Lambda}}{\BPS{\F}{\Lambda}}$.
\end{corollary}
\begin{proof}
The isomorphism $\psi_h$ is a consequence of \cref{thm:conjugate_actions_have_isomorphic_groupoids} and \cref{thm:conjugacy}. 
For the rest of the claim, it remains to show $h^{-1}(\BPS{\F}{\Lambda}) = \BPS{\M}{\Lambda}$ (because of the last part of \cref{thm:conjugate_actions_have_isomorphic_groupoids}).
Equivalently, we want to show $h(\BPS{\M}{\Lambda}) = \BPS{\F}{\Lambda}$.
Fix $x \in \BPS{\M}{\lambda}$.
Let $\mu \in h(x)$. 
We need to show $\mu$ is extendable in $h(x)$ (defined in \cref{def:extendable}).
Suppose $E \subseteq \Lambda$ is finite, exhaustive and $\source(\mu) \in \range(E)$.
We need to find $\nu \in E$ such that $\mu\nu \in h(x)$.
By definition of $h$, $h(x) = x(e\morphdom{x})$, so $\mu = x(e, m)$ for some $(e, m) \in e\morphdom{x}$.
Thus, $\source(x(e, m)) \in \range(E)$, and so there is a $\nu \in E$ such that $x(m, m\degree(\nu)) = \nu$ by definition of $\BPS{\M}{\Lambda}$.
Then, 
\[
x(e, m\degree(\nu)) 
= x(e, m)x(m, m\degree(\nu)) 
= \mu\nu 
\in h(x),
\] 
as required. 
That is, $h(x) \in \BPS{\F}{\Lambda}$.

For the other inclusion, fix $y \in \BPS{\F}{\Lambda}$.
Since $h$ is surjective, $y = h(x)$ for some $x \in \PS{\M}{\Lambda}$.
We want $x \in \BPS{\M}{\Lambda}$.
Suppose $(e, m) \in e\morphdom{x}$ and $E \subseteq \Lambda$ is finite, exhaustive and $\source(x(e, m)) \in \range(E)$.
Notice $x(e, m) \in h(x)$, so by definition of $\BPS{\F}{\Lambda}$ there is some $\nu \in E$ such that $x(e, m)\nu \in h(x) = x(e\morphdom{x})$.
Thus, $x(e, m)\nu = x(e, l) = x(e, m)x(m, l)$ for some $(e, l) \in e\morphdom{x}$, and so $\nu = x(m, l)$.
Because $x$ is degree-preserving, $x(e, m)\nu = x(e, l)$ implies $m\degree(\nu) = l$, so $\nu = x(m, m\degree(\nu))$.
Therefore, $x \in \BPS{\M}{\Lambda}$, as needed. 
That is, $h(\BPS{\M}{\Lambda}) = \BPS{\F}{\Lambda}$.
\end{proof}

We have now established our main results. In the rest of the paper, we reconcile the path spaces and groupoids of \cref{cor:isom} with other groupoids in the literature at various levels of generality (see \cref{rem:gpd_isoms_for_Pgraphs} for path groupoids of finitely aligned $P$-graphs, \cref{rem:gpd_isoms_for_Pgraphs_row_finite} for path groupoids of row-finite $P$-graphs, \cref{rem:space_homeos_for_kgraphs} for path \emphasis{spaces} of finitely aligned $k$-graphs, and \cref{rem:gpd_isoms_for_kgraphs} for path groupoids of finitely aligned 
$k$-graphs).

\subsection{Reconciling with other groupoids of \texorpdfstring{$P$-graphs}{P-graphs}}

Recall from \cref{sec:Pgraphs} that any $P$-graph $\Lambda$ is a category of paths, and so we may also consider Spielberg's groupoids from \cite{Spi20}. 
By \cite[Proposition 5.11]{Spi20}, the two groupoids from \cite[Definition 5.5(iii)]{Spi20} are isomorphic. 
We denote this groupoid by $\PG{\ofSpielberg}{\Lambda}$.
Spielberg defines a subspace of its unit space in \cite[Definition 10.2]{Spi20}, and we denote the reduction of $\PG{\ofSpielberg}{\Lambda}$ to this subspace by $\BPG{\ofSpielberg}{\Lambda}$.
Combining \cref{cor:isom} with \cite[Corollary 4.14 and Proposition 5.2]{OP20} and \cite[Theorem 6.6 and Corollary 6.8]{Jon26}, we have the following remark.

\begin{remark}
\label{rem:gpd_isoms_for_Pgraphs}
For any finitely aligned $P$-graph $\Lambda$, where $(Q, P)$ is a weakly quasi-lattice ordered group,
\begin{align*}
&\PG{\ofSpielberg}{\Lambda} 
\cong \PG{\M}{\Lambda} 
\cong \PG{\F}{\Lambda} \text{ and } \\
&\BPG{\ofSpielberg}{\Lambda}
\cong \BPG{\M}{\Lambda}
\cong \BPG{\F}{\Lambda} 
\cong \tightgpd(S_{\Lambda, \ofOP}),
\end{align*}
for the inverse semigroup $S_{\Lambda, \ofOP}$ from \cite[Definition 2.2]{OP20} and Exel's tight groupoid $\tightgpd(S_{\Lambda, \ofOP})$ from \cite{Exe08}.
\end{remark}

\begin{remark}
By \cite[Lemma 4.4]{OP20}, the space $\hat{E}_*(S_{\Lambda, \ofOP})$ of characters of the idempotent semilattice $E(S_{\Lambda, \ofOP})$ that satisfy condition $(*)$ in the sense of \cite[Definition 3.15]{OP20} is invariant under the natural action of $S_{\Lambda, \ofOP}$. 
Moreover, by \cite[Corollary 3.16]{OP20}, $\hat{E}_*(S_{\Lambda, \ofOP})$ is homeomorphic to $\PS{\F}{\Lambda}$. 
It would be worth checking that the groupoid of germs of the action of $S_{\Lambda, \ofOP}$ on $\hat{E}_*(S_{\Lambda, \ofOP})$ coincides with the path groupoids of \cref{rem:gpd_isoms_for_Pgraphs}.
\end{remark}

\begin{remark}
\label{rem:gpd_isoms_for_Pgraphs_row_finite}
In the setting of \cref{rem:gpd_isoms_for_Pgraphs}, if $\Lambda$ is \Define{row-finite} (i.e. for any $(v, p) \in \unitspace{\Lambda} \times P$, the set $\setof{\lambda \in \Lambda}{\range(\lambda) = v \text{ and } \degree(\lambda) = p}$ is finite), then we also have $\PG{\F}{\Lambda}$ is isomorphic to the Toeplitz groupoid from \cite[Definition 6.11]{RW17} and $\BPG{\F}{\Lambda}$ is isomorphic to the Cuntz--Krieger groupoid from \cite[Definition 6.18]{RW17} as per \cite[Remark 6.7]{Jon26}.
\end{remark}

\begin{remark}
We remark on the names ``Toeplitz groupoid'' and ``Cuntz--Krieger groupoid'' in \cref{rem:gpd_isoms_for_Pgraphs_row_finite}.
As in the directed graph setting, each finitely aligned $k$-graph has both a Toeplitz and Cuntz--Krieger algebra \cite[Remark 3.9]{FMY05}. 
The C*-algebras associated to the path and boundary-path groupoids according to \cite{Ren80} coincide with the Toeplitz and Cuntz--Krieger algebras, respectively. 
The Cuntz--Krieger algebra is often simply called the ``graph C*-algebra'' and the boundary-path groupoid is called the ``graph groupoid'' or even the ``path groupoid''.
That is, the terminology for such groupoids is inconsistent, so we presume \cite{RW17} have used the terms ``Toeplitz groupoid'' and ``Cuntz--Krieger groupoid'' to avoid confusion.
\end{remark}

\section{\texorpdfstring{$k$-graphs}{k-graphs}}
\label{sec:k}

Throughout this section, let $\Lambda$ be a finitely aligned $k$-graph.
We write $m(i)$ for the $i$th entry of the $k$-tuple $m \in \N^k$.
We show the $P$-path prototypes $\Omega_{P, (m_n)}$ defined in \cref{prop:increasing_sequences_in_P_yield_nested_P_graphs} coincide with the rank-$k$ path prototypes $\Omega_{k, m}$ of \cite{Yee07}.
Given $\Omega_{k, m}$ for some $m \in (\N \cup \{\infty\})^k$, we have that $\Omega_{k, m} = \Omega_{\N^k, (m_n)}$, where 
\[
(m_n(i)) \coloneqq 
\begin{cases}
(m(i), m(i), \dots) & \text{ if } m(i) \in \N, \\
(0, 1, 2, \dots) & \text{ if } m(i) = \infty,
\end{cases}
\]
for each $1 \leq i \leq k$.
Also, given $\Omega_{\N^k, (m_n)}$ for some $\leq$-increasing $(m_n) \subseteq \N^k$, we have that $\Omega_{k, m} = \Omega_{\N^k, (m_n)}$, where 
\[
m(i) \coloneqq 
\begin{cases}
\max\setof{m_n(i)}{n \in \N} & \text{ if $(m_n(i))$ is bounded,} \\
\infty & \text{ if $(m_n(i))$ is unbounded},
\end{cases}
\]
for each $1 \leq i \leq k$.
Therefore,
\(
\setof{\Omega_{k, m}}{m \in (\N \cup \{\infty\})^k}
= \setof{
\Omega_{\N^k, (m_n)}
}{
(m_n) \subseteq \N^k \text{ is $\leq$-increasing}
}.
\)
Hence, $\PS{\M}{\Lambda}$ coincides with $\PS{\ofYeend}{\Lambda}$ from \cite[Definition 3.1]{Yee07}. 
Since \cite[Examples 4.10(ii)]{Yee07} reconciles the path spaces and groupoids in \cite{Yee07} with those in \cite{FMY05}, and by comparing our definitions of exhaustive sets and $\BPS{\M}{\Lambda}$ with those of \cite{FMY05}, we make the following remark (which adds to the existing homeomorphisms of path and boundary-path spaces from \cref{rem:gpd_isoms_for_Pgraphs} for $P$-graphs).

\begin{remark}
\label{rem:space_homeos_for_kgraphs}
For any finitely aligned (discrete) $k$-graph, 
\[
\PS{\ofFMY}{\Lambda}
\cong \PS{\ofYeend}{\Lambda}
\cong \PS{\ofSpielberg}{\Lambda}
\cong \PS{\M}{\Lambda} 
\cong \PS{\F}{\Lambda}
\cong \hat{E}_*(S_{\Lambda, \ofOP}),
\]
where $\hat{E}_*(S_{\Lambda, \ofOP})$ is the space of characters of the idempotent semilattice $E(S_{\Lambda, \ofOP})$ that satisfy condition $(*)$ as per \cite[Definition 3.15 and Corollary 3.16]{OP20}, and
\[
\BPS{\ofFMY}{\Lambda}
\cong \BPS{\ofYeend}{\Lambda}
\cong \BPS{\ofSpielberg}{\Lambda}
\cong \BPS{\M}{\Lambda} 
\cong \BPS{\F}{\Lambda}
\cong \hat{E}_\mathrm{tight}(S_{\Lambda, \ofOP}),
\]
where $\hat{E}_\mathrm{tight}(S_{\Lambda, \ofOP}) \cong \unitspace{\tightgpd(S_{\Lambda, \ofOP})}$ is the space of characters of $E(S_{\Lambda, \ofOP})$ that are tight as per \cite[Definition 12.8]{Exe08}.
\end{remark}

Now, we check that our action $(\PS{\M}{\Lambda}, P, T_{\PS{\M}{\Lambda}})$ coincides with the action underlying the path groupoid $\PG{\ofYeend}{\Lambda}$ so that $\PG{\M}{\Lambda} = \PG{\ofYeend}{\Lambda}$.
Each $x \in \PS{\M}{\Lambda}$ is of the form $x \colon \Omega_{k, m} \to \Lambda$ for some $m \in (\N \cup \{\infty\})^k$.
Then, $(x, p) \in \PS{\M}{\Lambda} * \N^k$ if and only if $(e, p) \in \Omega_{k, m}$ if and only if $p \leq m$, so the domain of the action of each $p \in \N^k$ in our construction equals the domain of the action of $p$ underlying $\PG{\ofYeend}{\Lambda}$.
In the above setting, following \cite[Lemma 3.3]{Yee07}, $\sigma^p x$ denotes the unique element of $\PS{\M}{\Lambda}$ with $\morphdom{\sigma^p x} = \Omega_{k, m - p}$ and 
\(
(\sigma^p x)(0, q) 
= x(p, p + q),
\)
for each $q \leq m - p$. 
The graph morphism $\RPA{x}{m}$ too has these properties by construction, so $\RPA{x}{m} = \sigma^m x$.
Thus, the actions are the same, and so $\PG{\M}{\Lambda} = \PG{\ofYeend}{\Lambda}$ as sets.
Moreover, the groupoid operations are the same, so they are equal as groupoids too.

It remains to show the topology $\tau_\ofYeend$ on $\PG{\ofYeend}{\Lambda}$ equals the topology $\tau$ on $\PG{\M}{\Lambda}$ so that $\PG{\M}{\Lambda} = \PG{\ofYeend}{\Lambda}$ as topological groupoids. 
This is well-known, but we are not aware of a reference where the details have been given.
We start by identifying a basis for each topology.
By \cref{cor:isom} and \cite[Lemma 6.3]{Jon26}, the collection of 
\[
\Fellgpd{\PG{\M}{\Lambda}}{\kappa}{K}{\lambda}{L} \coloneqq 
\SDPGcylinder{\PS{\M}{\Lambda}}{
\Fell{\PS{\M}{\Lambda}}{\kappa}{K}
}{
\degree(\kappa)
}{
\degree(\lambda)
}{
\Fell{\PS{\M}{\Lambda}}{\lambda}{L}
},
\]
where $\kappa, \lambda \in \Lambda$ and $K \subseteq \kappa\Lambda, L \subseteq \lambda\Lambda$ are finite, is a basis for $\tau$.
Given any sets $A, B \subseteq \Lambda$, we write 
\[
A *_{\source} B \coloneqq 
\setof{
(\lambda, \mu) \in A \times B
}{
\source(\lambda) = \source(\mu)
}.
\]
For each $m \in \Z^k$ and $F \subseteq \Lambda *_{\source} \Lambda$, define
\[
Z_\ofYeend(F, m) \coloneqq 
\setof{
(\lambda x, \degree(\lambda) - \degree(\mu), \mu x) \in \PG{\M}{\Lambda}
}{
(\lambda, \mu) \in F, \degree(\lambda) - \degree(\mu) = m
},
\]
where $\lambda x$ is defined as in \cite[Lemma 3.3]{Yee07}. 
We can assume that each $(\lambda, \mu) \in F$ satisfies $\degree(\lambda) - \degree(\mu) = m$.
By \cite[Proposition 3.6]{Yee07}, the collection of sets of the form 
\(
Z_\ofYeend(A *_{\source} B, m) \cap \compl{Z_\ofYeend(F, m)}
\)
for $m \in \Z^k$ and $A, B \subseteq \Lambda$ and finite $F \subseteq \Lambda *_{\source} \Lambda$ is a basis for $\tau_\ofYeend$.

We show $\tau_\ofYeend \subseteq \tau$.
For any $m \in \N^k$ and $A, B \subseteq \Lambda$,
\[
Z_\ofYeend(A *_{\source} B, m) 
= \bigcup_{(\lambda, \mu) \in A *_{\source} B}
\Fellgpd{\PG{\M}{\Lambda}}{\lambda}{\emptyset}{\mu}{\emptyset}
\in \tau.
\]
Now fix $m \in \Z^k$ and finite $F \subseteq \Lambda *_{\source} \Lambda$.
To show $\compl{Z_\ofYeend(F, m)} \in \tau$, we show $Z_\ofYeend(F, m)$ is closed with respect to $\tau$. 
Suppose $(g_n) \subseteq Z_\ofYeend(F, m)$ and $g_n \to g \in \PG{\M}{\Lambda}$ with respect to $\tau$. 
Because $(g_n) \subseteq Z_\ofYeend(F, m)$, we can write $g_n = (\lambda_n x_n, m, \mu_n x_n)$, where $(\lambda_n, \mu_n) \in F$, for all $n$.
Because $F$ is finite, by the pigeon-hole principle, there is a subsequence $(g_{n_k})$ and some $(\lambda, \mu) \in F$ such that $g_{n_k} = (\lambda x_{n_k}, m, \mu x_{n_k})$, for all $k$.
Moreover, $g_{n_k} \to g$. 
Thus, we can assume without loss of generality that $g_n = (\lambda x_n, m, \mu x_n)$, for all $n$.
Say $g = (w', m', z')$.
By \cite[Proposition 5.12(b)]{RW17}, the map $(x, n, z) \mapsto n$ from $\PG{\M}{\Lambda}$ to $\Z^k$ is continuous with respect to $\tau$ and the discrete topology on $\Z^k$, so $(m, m, \dots)$ converges to $m'$, which means $m = m'$.
Since $\PG{\M}{\Lambda}$ is a topological groupoid, $\range(g_n) \to \range(g)$ and $\source(g_n) \to \source(g)$ with respect to $\tau$.
The unit space $\unitspace{\PG{\M}{\Lambda}}$ is homeomorphic to $\PS{\M}{\Lambda}$ so that $\lambda x_n \to w'$ and $\mu x_n \to z'$.
Notice $(\lambda x_n) \subseteq \Fellin{\PS{\M}{\Lambda}}{\lambda}$ and $(\mu x_n) \subseteq  \Fellin{\PS{\M}{\Lambda}}{\mu}$, which are closed in $\PS{\M}{\Lambda}$. 
Hence, $w' = \lambda w$ and $z' = \mu z$ for some $w, z \in \PS{\M}{\Lambda}$.
That is, $\lambda x_n \to \lambda w$ and $\mu x_n \to \mu z$.
Both $\sigma^\lambda$ and $\sigma^\mu$ are continuous, and so $x_n \to w, z$, which implies $x \coloneqq w = z$. 
Thus, $g = (\lambda x, m, \mu x) \in Z_\ofYeend(F, m)$, and so $Z_\ofYeend(F, m)$ is closed, as required.

Now, we show $\tau \subseteq \tau_\ofYeend$.
We show each $B \coloneqq \Fellgpd{\PG{\M}{\Lambda}}{\kappa}{K}{\lambda}{L} \in \tau_\ofYeend$.
If $(\kappa, \lambda) \notin \Lambda *_{\source} \Lambda$, then $B = \emptyset \in \tau_\ofYeend$.
Suppose $(\kappa, \lambda) \in \Lambda *_{\source} \Lambda$.
Let $F \coloneqq \setof{(\kappa\zeta, \lambda\zeta) \in \Lambda *_{\source} \Lambda}{\kappa\zeta \in K \text{ or } \lambda\zeta \in L}$.
We show 
\[
\Fellgpd{\PG{\M}{\Lambda}}{\kappa}{K}{\lambda}{L}
= Z_\ofYeend(\{\kappa\} *_{\source} \{\lambda\}, \degree(\kappa) - \degree(\lambda)) 
\cap \compl{Z_\ofYeend(F, \degree(\kappa) - \degree(\lambda))}.
\]
Fix $g \coloneqq (\kappa x, \degree(\kappa) - \degree(\lambda), \lambda x) \in \Fellgpd{\PG{\M}{\Lambda}}{\kappa}{K}{\lambda}{L}$.
Then, $g \in Z_\ofYeend(\{\kappa\} *_{\source} \{\lambda\}, \degree(\kappa) - \degree(\lambda))$.
Suppose for a contradiction that $g \in Z_\ofYeend(F, \degree(\kappa) - \degree(\lambda))$.
Then, $g = (\kappa\zeta w, \degree(\kappa) - \degree(\lambda), \lambda\zeta w)$ for some $(\kappa\zeta, \lambda\zeta) \in F$.
By definition of $F$, $\kappa\zeta \in K$ or $\lambda\zeta \in L$.
In either case, there is a contradiction with $g \in \Fellgpd{\PG{\M}{\Lambda}}{\kappa}{K}{\lambda}{L}$.
For the other inclusion, fix $g \in Z_\ofYeend(\{\kappa\} *_{\source} \{\lambda\}, \degree(\kappa) - \degree(\lambda)) \cap \compl{Z_\ofYeend(F, \degree(\kappa) - \degree(\lambda))}$, so $g = (\kappa x, \degree(\kappa) - \degree(\lambda), \lambda x)$ for some $x \in \PS{\M}{\Lambda}$.
Then, $g \in \Fellgpd{\PG{\M}{\Lambda}}{\kappa}{\emptyset}{\lambda}{\emptyset}$.
Suppose for a contradiction that $\kappa x \notin \Fell{\PS{\M}{\Lambda}}{\kappa}{K}$ or $\lambda x \notin \Fell{\PS{\M}{\Lambda}}{\lambda}{L}$.
In the former case, we have $\kappa x = \kappa\zeta w$ for some $\kappa\zeta \in K$, in which case $(\kappa\zeta, \lambda\zeta) \in F$, and $x = \zeta w$.
Thus, $g = (\kappa\zeta w, \degree(\kappa) - \degree(\lambda), \lambda\zeta w)$, so $g \in Z_\ofYeend(F, \degree(\kappa) - \degree(\lambda))$, contradicting $g \in \compl{Z_\ofYeend(F, \degree(\kappa) - \degree(\lambda))}$.
In the latter case, we have a similar contradiction, and so $g \in \Fellgpd{\PG{\M}{\Lambda}}{\kappa}{K}{\lambda}{L}$.

Therefore, $\PG{\M}{\Lambda} = \PG{\ofYeend}{\Lambda}$ as topological groupoids. 
Since $\BPS{\M}{\Lambda} = \BPS{\ofYeend}{\Lambda}$, we also have that
\(
\BPG{\M}{\Lambda}
= \reduction{\PG{\M}{\Lambda}}{\BPS{\M}{\Lambda}}
= \reduction{\PG{\ofYeend}{\Lambda}}{\BPS{\ofYeend}{\Lambda}}
= \BPG{\ofYeend}{\Lambda}.
\)

\begin{remark}
\label{rem:gpd_isoms_for_kgraphs}
Recall the list of isomorphic path groupoids associated to $P$-graphs in \cref{rem:gpd_isoms_for_Pgraphs}.
By \cite[Examples 4.10(ii)]{Yee07} and the above discussion, we now have the following.
For any finitely aligned $k$-graph $\Lambda$,
\begin{align*}
&\PG{\ofFMY}{\Lambda}
\cong \PG{\ofYeend}{\Lambda}
\cong \PG{\ofSpielberg}{\Lambda} 
\cong \PG{\M}{\Lambda} 
\cong \PG{\F}{\Lambda} \text{ and } \\
&\BPG{\ofFMY}{\Lambda}
\cong \BPG{\ofYeend}{\Lambda}
\cong \BPG{\ofSpielberg}{\Lambda}
\cong \BPG{\M}{\Lambda}
\cong \BPG{\F}{\Lambda} 
\cong \tightgpd(S_{\Lambda, \ofOP}),
\end{align*}
for the inverse semigroup $S_{\Lambda, \ofOP}$ from \cite[Definition 2.2]{OP20} and Exel's tight groupoid $\tightgpd(S_{\Lambda, \ofOP})$ from \cite[Theorem 13.3]{Exe08}.
\end{remark}

\begin{remark}
By \cite[Remark 5.9]{FMY05}, for any finitely aligned $k$-graph $\Lambda$, the C*-algebras of the path and boundary-path groupoids of \cref{rem:gpd_isoms_for_kgraphs} are isomorphic to the Toeplitz and Cuntz--Krieger algebras of $\Lambda$ as in \cite[Remark 3.9]{FMY05}.
\end{remark}

\begin{remark}
\label{rem:Exel}
\cref{rem:gpd_isoms_for_kgraphs} implies $\BPG{\ofFMY}{\Lambda}$ is isomorphic to $\tightgpd(S_{\Lambda, \ofOP})$. 
In this way, we give a partial answer to Exel's conjecture on \cite[Page 195]{Exe08}, as discussed in the introduction. 
The answer is \emphasis{partial} in the sense that the inverse semigroup $S_{\Lambda, \ofFMY}$ from \cite[Definition 4.1]{FMY05} is not isomorphic to $S_{\Lambda, \ofOP}$. 
Roughly speaking, $S_{\Lambda, \ofFMY}$ contains joins unlike $S_{\Lambda, \ofOP}$, but some investigation is needed to relate the two inverse semigroups precisely and determine whether they have the same tight groupoid. 
Following \cite{arXiv_Exe25}, $S_{\Lambda, \ofFMY}$ and $S_{\Lambda, \ofOP}$ have the same tight groupoid if and only if they are \emphasis{consonant}.
\end{remark}

\appendix 
\section{Notation for path spaces and groupoids}
\label{sec:notation}

The notation of path spaces and their groupoids is inconsistent in the literature. 
The convention we have chosen for this paper is as follows:

\begin{itemize}
\item Path spaces
\begin{itemize}
\item $\PS{\F}{\Lambda}$ \dotfill \cref{def:PSfilter}
\item $\PS{\M}{\Lambda}$ \dotfill \cref{def:PSmorph}
\item $\PS{\ofFMY}{\Lambda}$ \dotfill $X_\Lambda$ of \cite[Definition 5.1]{FMY05}
\item $\PS{\ofYeend}{\Lambda}$ \dotfill $X_\Lambda$ of \cite[Definition 3.1]{Yee07}
\item $\PS{\ofSpielberg}{\Lambda}$ \dotfill $X(\Lambda)$ of \cite[Page 1586]{Spi20}
\end{itemize}
\item Path groupoids
\begin{itemize}
\item $\PG{\F}{\Lambda}$ \dotfill \cref{def:PGfilter}
\item $\PG{\M}{\Lambda}$ \dotfill \cref{def:PGmorph}
\item $\PG{\ofFMY}{\Lambda}$ \dotfill $\mathcal{G}_\Lambda$ of \cite[\S 6]{FMY05}
\item $\PG{\ofYeend}{\Lambda}$ \dotfill $G_\Lambda$ of \cite[Definition 3.4]{Yee07}
\item $\PG{\ofSpielberg}{\Lambda}$ \dotfill $G(\Lambda)$ of \cite[Definition 5.5]{Spi20}
\end{itemize}
\item Boundary-path spaces
\begin{itemize}
\item $\BPS{\F}{\Lambda}$ \dotfill \cref{def:BPSfilter}
\item $\BPS{\M}{\Lambda}$ \dotfill \cref{def:BPSmorph}
\item $\BPS{\ofFMY}{\Lambda}$ \dotfill $\partial\Lambda$ of \cite[Definition 5.10]{FMY05}
\item $\BPS{\ofYeend}{\Lambda}$ \dotfill $\partial\Lambda$ of \cite[Definition 4.2]{Yee07}
\item $\BPS{\ofSpielberg}{\Lambda}$ \dotfill $\partial\Lambda$ of \cite[Definition 10.2]{Spi20}
\end{itemize}
\item Boundary-path groupoids
\begin{itemize}
\item $\BPG{\F}{\Lambda}$ \dotfill \cref{def:BPGfilter}
\item $\BPG{\M}{\Lambda}$ \dotfill \cref{def:BPGmorph}
\item $\BPG{\ofFMY}{\Lambda}$ \dotfill $\mathcal{G}_\Lambda|_{\partial\Lambda}$ of \cite[Page 181]{FMY05}
\item $\BPG{\ofYeend}{\Lambda}$ \dotfill $\mathcal{G}_\Lambda$ of \cite[Definition 4.8]{Yee07}
\item $\BPG{\ofSpielberg}{\Lambda}$ \dotfill $G(\Lambda)|_{\partial\Lambda}$ of \cite[Page 1602]{Spi20}
\end{itemize}
\end{itemize}

\section*{Acknowledgements}

This research was supported by Marsden grant 21-VUW-156 from the Royal 
Society of New Zealand and partially supported by a grant from the Simons 
Foundation. 
The authors would like to thank the Isaac Newton Institute for Mathematical 
Sciences for the support and hospitality during the programme `Topological 
groupoids and their C*-algebras' when work on this paper was undertaken. 
This work was supported by EPSRC Grant Number EP/V521929/1. 
The second author was supported by grant P1-0288 from the Slovenian Research 
and Innovation Agency. 
The second author would like to thank Astrid an Huef for her supervision 
while undertaking this research as a PhD student.
The authors would also like to thank the referee for suggestions that 
improved the presentation of our paper.

\newcommand{\etalchar}[1]{$^{#1}$}

\end{document}